\documentclass[11pt,a4paper]{amsart}
\usepackage{mathrsfs}
\usepackage{amsfonts}
 \usepackage[colorlinks=true, pdfstartview=FitV, linkcolor=blue, citecolor=blue, urlcolor=blue]{hyperref}

\usepackage[active]{srcltx}
\usepackage{enumerate}

\usepackage{amsmath,amssymb,xspace,amsthm}

\newtheorem{theorem}{Theorem}
\newtheorem{lemma}[theorem]{Lemma}
\newtheorem{proposition}[theorem]{Proposition}
\numberwithin{equation}{section}
\newcommand{\thmref}[1]{Theorem~\ref{#1}}

\newcommand{\eqnref}[1]{~(\ref{#1})}

\def\Der{\operatorname{Der}}

\newcommand{\C}{\ensuremath{\mathbb C}\xspace}

\renewcommand{\i}{\ensuremath{\mathfrak i}}

\newcommand{\N}{\mathbb{N}}
\newcommand{\Z}{\ensuremath{\mathbb{Z}}\xspace}

\renewcommand{\phi}{\varphi}

\renewcommand{\leq}{\leqslant}
\renewcommand{\geq}{\geqslant}

\newcommand{\p}{\partial}

\title[Polynomials from hyperelliptic Lie algebras]{Certain families of Polynomials arising in the study of hyperelliptic Lie algebras}

\author{Ben Cox}\author{Kaiming Zhao}
\address{Department of Mathematics \\
University of Charleston, South Carolina \\
66 George Street  \\
Charleston SC 29424, USA}\email{coxbl@cofc.edu}

\address{Department of Mathematics\\
 Wilfrid
Laurier University \\ Waterloo, ON \\
Canada N2L 3C5\\
 and College of
Mathematics and Information Science \\
Hebei Normal (Teachers)
University \\  Shijiazhuang, Hebei, 050016  \\ P. R. China. }
\email{kzhao@wlu.ca}


\begin{document}
\begin{abstract}
The associative ring $R(P(t))=\C[t^{\pm1},u \,|\, u^2=P(t)]$, where $P(t)=\sum_{i=0}^na_it^i=\prod_{k=1}^n(t-\alpha_i)$ with $\alpha_i\in\mathbb C$ pairwise distinct, is the coordinate ring of a hyperelliptic curve.  The Lie algebra  $\mathcal{R}(P(t))=\Der(R(P(t)))$ of derivations is called the hyperelliptic Lie algebra associated to $P(t)$.
 In this paper we  describe the universal central extension of $\Der(R(P(t)))$  in terms of certain families of polynomials which in a particular case are associated Legendre polynomials.  Moreover we describe certain families of polynomials that arise in the study of the group of units for the ring $R(P(t))$ where $P(t)=t^4-2bt^2+1$.   In this study pairs of Chebychev polynomials $(U_n,T_n)$ arise as particular cases of a pairs $(r_n,s_n)$ with $r_n+s_n\sqrt{P(t)}$ a unit in $R(P(t))$.  We explicitly describe these polynomial pairs as coefficients of certain  generating functions and show certain of these polynomials satisfy particular second order  linear differential equations.
\end{abstract}
\maketitle

\section{Introduction}

Throughout this paper we will take the set of natural numbers to be $\mathbb N=\{1,2,\dots\}$, the set of nonnegative integers will be denoted by $\mathbb
Z_+=\{0,1,2,3,\dots\}$, and we will assume all vector spaces and algebras are defined over
complex numbers $\C$.

The associative ring $R(P(t))=\C[t^{\pm1},u \,|\, u^2=P(t)]$, where $P(t)=\sum_{i=0}^na_it^i=\prod_{k=1}^n(t-\alpha_i)$ with $\alpha_i\in\mathbb C$ pairwise distinct, is the coordinate ring of a hyperelliptic curve.  The Lie algebra  $\mathcal{R}(P(t))=\Der(R(P(t)))$ of derivations is called the hyperelliptic Lie algebra associated to $P(t)$.
 In this paper we  describe the universal central extension of $\Der(R(P(t)))$  in terms of certain families of polynomials which in a particular case are associated Legendre polynomials.  Moreover we describe certain families of polynomials that arise in the study of the group of units for the ring $R(P(t))$ where $P(t)=t^4-2bt^2+1$.   In this study pairs of Chebychev polynomials $(U_n,T_n)$ arise as particular cases of a pairs $(r_n,s_n)$ with $r_n+s_n\sqrt{P(t)}$ a unit in $R(P(t))$.  We explicitly describe these polynomial pairs as coefficients of certain  generating functions and show certain of these polynomials satisfy particular second order  linear differential equations.

Let us give some background to this paper and more precise information of what is in the paper. The Laurent polynomial ring $\C[t, t^{-1}]$ can be considered as the ring
of rational functions on the Riemann sphere $\C \cup \{\infty\}$
with poles allowed only in $\{\infty, 0\}$. From geometric point of
view one can have  a natural generalization of the loop algebra
construction. Instead of the sphere with two punctures, one can
use any complex algebraic curve $X$ of genus $g$ with a fixed
subset $P$ of $n$ distinct points. This arrives at M. Schlichenmaier's definition of multipoint algebras of Krichever-Novikov affine type if we replace $\C[t, t^{-1}]$ with the ring $R$ of
meromorphic functions on $X$ with poles allowed only in $P$ in the
construction of affine Kac-Moody algebras (see \cite{MR2058804},  \cite{MR902293}, \cite{MR925072}, and \cite{MR998426}). The $n$-point affine Lie algebras which are a type of Krichever-Novikov algebra of genus zero also appeared in the
work of Kazhdan and Lusztig (\cite[Sections 4 \& 7]{MR1104840},\cite[Chapter 12]{MR1849359}).  Krichever-Novikov algebras are used to constuct analogues of important mathematical objects used in string theory but in the setting of a Riemann surface of arbitrary genus. Moreover Wess-Zumino-Witten-Novikov theory and analogues of the Knizhnik-Zamolodchikov equations are developed for analogues of the affine and Virasoro algebras (see the survey article \cite{MR2152962}, and for example  \cite{MR1706819},  \cite{MR2072650}, \cite{MR2058804}, \cite{MR1989644}, and \cite{MR1666274}).

In a recent paper \cite{CGLZ1}   the $n$-point Virasoro algebras $\tilde{\mathcal{V}}_a$ were studied,  which are natural generalizations
 of the classical Virasoro algebra and have as quotients multipoint genus zero Krichever-Novikov type algebras. Necessary and sufficient conditions for the latter two such Lie algebras
 to be isomorphic, and   their automorphisms, their derivation algebras were obtained, their universal
 central extensions,  and some other
 properties were also determined.   Also a large class of modules which were called modules of densities were constructed, and
  necessary and sufficient conditions for them to be irreducible were obtained.  The $n$-point Virasoro algebras all are coordinate rings of the rings the Riemann sphere with $n$-points removed.  

In another recent paper \cite{CGLZ2}, $m$-th
superelliptic Lie algebras $\mathcal{R}_m(P)$ associated to $P(t)\in\C[t]$ were defined,
  the necessary and sufficient conditions for such Lie
algebras to be simple were obtained, and   their universal central extensions and their derivation algebras were determined.  They are examples of genus greater than one Krichever-Novikov type algebras.

The present paper is a sequel to \cite{CGLZ2}.    The particular class we look at are  hyperelliptic curves, i.e., $m=2$.  Let $P(t)\in\C[t]$. Then we have the Riemann surfaces (commutative associative algebras) $R(P)=\C[t^{\pm1},u \,|\, u^2=P(t)].$ The Lie algebras  $\mathcal{R}(P)=\Der(R(P))$ are called  {\it  hyperelliptic Lie
algebras} due to the fact that $u^2=P(t)$ is a  hyperelliptic curve.  

Finally we should mention that other interesting families of non-classical orthogonal polynomials appear in the description of the center of the universal enveloping algebra of other particular Krichever-Novikov algebras such as the DKJM algebra (see \cite{MR3090080}).

 In the second section, we give a basis for the universal central extension of the Lie algebras
 $\mathcal{R}_m(P)$ and in the particular case that $0$ is not a root of $P(t)$ explicitly describe a set of basis of two cocycles. We also explicitly give examples the value of any 2-cocyle on a basis of $\mathcal{R}_m(P)$.    Our description uses Bra\'a di Bruno's formula and Bell polynomials to describe generating functions of families of polynomials appearing as coefficients of the basis elements.   In the particular case of $P(t)=t^{2r}-2bt^r+1$, $r\in\mathbb N$, associated Legendre polynomials naturally arise.
 Having such an explicit description of the two cocycles will allow, one using conformal field theoretic tools, to study free field type representations of these algebras.

To study isomorphisms and automorphisms between the Lie algebras $\mathcal{R}_m({P})$, from \cite{MR966871} we know that it is  equivalent  to considering isomorphisms and automorphisms
 of the Riemann surfaces  ${R}_m({P})$. This is in general a very hard problem. In particular it is known that many sporadic simple groups can appear. For some results on higher genus Riemann surfaces, see \cite{MR2203507, MR1796706} and the reference therein.

  In the third section, we describe the group of units of $R_2(P)$  which will be used in the last section,  in particular we explicitly describe this group in cases of$P(t)=t^4-2bt+1$, $b\neq \pm 1$ which is the most interesting case studied by Date, Jimbo, Kashiwara and Miwa \cite{DJKM}.  In this later paper they investigated integrable systems arising from Landau-Lifshitz differential equation.  When determining the unit group we find that it requires one find solutions of the polynomial Pell equation $f^2-g^2P=1$
for a given $P\in\C[t]$ which is a very famous and very hard problem.
See \cite{MR2183270}.
  Then \thmref{CGLZ2} describes the possible units in $\mathcal{R}_m(P)$ as $\mathcal{R}_m(P)^*\cong \mathbb C^*\times \mathbb Z\times \mathbb Z\times \mathbb Z$.  One may want to compare this to the description of the automorphism group of a hyperellptic curve given in \cite{MR2035219}.  The group of units of $R_2(P)$ are described in terms of pairs of polynomials $u_n$, $v_n$ satisfying $u_n+v_n\sqrt{P}=(u_1+v_0\sqrt{P})^n$ and are similar in some sense to the Chebyshev polynomials of the first and second kinds.   We explicitly describe these polynomial pairs as coefficients of generating function and show certain of these polynomials satisfy particular second order  linear differential equations.

\section{The hyperelliptic Lie algebras $\mathcal{R}(P(t))$}
Let  $P(t)=\sum_{i=l}^{n+l}b_it^i=t^l(t-a_1)\ldots(t-a_n)$ for some $l=0$ or $1$, $n\in \N$ and pairwise distinct nonzero $a_1,\ldots,a_n\in \C.$ Note that $b_{n+l}=1$.
It is easy to see from \cite{CGLZ2} that
$R(P)=\C[t^{\pm1}]\oplus  \C[t^{\pm1}]u$ and if we set  $\Delta:=P'\frac \p{\p u}+2u\frac \p{\p t}$
where $P'=\frac {\p P}{\p t}$ we have $\mathcal{R}(P)=R(P)\Delta$.
One of the reasons we are interested in the Lie algebra $\mathcal{R}(P):=\text{Der}R(P(t))$ is the following

\begin{theorem}\cite[Theorem 4]{CGLZ2}  The Lie algebra $\mathcal{R}(P(t))$ is  simple.
\end{theorem}

 Let $R=R(P)$ and  $\partial=\sqrt{P}\frac{d}{dt}.$ We know

\begin{theorem} \cite[Theorem 9]{CGLZ2}  Suppose that $R=R (P)$ and
$P=t^l(t-a_1)\ldots(t-a_n)$ for $l=0$ or $1$, $n\in \N$ and pairwise distinct nonzero $a_1,\ldots,a_n\in\C$.
Then $\dim R/\partial(R)=n+1$ and the universal central extension of
$\mathcal{R}(P)$ is $\mathcal{R}(P)\oplus R/\p(R)$ with brackets
 $$[f\p,g\p]=f\p(g)\p-g\p(f)\p+\overline{\p(f)\p(\p(g))}, \forall\,\,f,g\in R.$$
\end{theorem}

Note that $\mathcal{R}(P)$ is the classical centerless Virasoro algebra if $n=0$, and $\mathcal{R}(P)$ with $l>1$ is isomorphic to one of the the Lie algebras in the above theorem. So our Lie algebra $\mathcal{R}(P)$ with the given restrictions on $P$ cover all cases of $P$ up to isomorphisms.

To study representations of the Lie algebras $\mathcal{R}(P)\oplus R/\p(R)$ with nontrivial central action, it is important to have a basis for its center $R/\p(R)$ with concrete formulae for their brackets. This is the main purpose of this section, which turns out to be complicated.

 From the proof of \cite[Theorem 9]{CGLZ2} we suppose a basis is given by $\omega_0:=\overline{t^{-1}\sqrt{P}}$,
and
$$
\{\overline{t^{-1}},\dots, \overline{t^{n-2}}\}.
 $$
 In the rest of this section we will derive some formulae for the 2-cocycles of $\mathcal{R}(P)$ in terms of this precise basis.

\subsection{The general case} We can calculate any 2-cocyle as being a linear combination of the coefficients of $\psi$ where $\psi(f\partial,g\partial)=\overline{\partial(f)\partial^2(g)}$.  In particular
\begin{align}
\psi(t^r\partial, t^s\partial)
&=rs\sum_{i=0}^n\left(s-1+\frac{i}{2}\right)b_i\delta_{r+s+i-2,0}\omega_0 \\
&=\begin{cases}
(rs(r-s)b_{-r-s+2}\omega_0)/2,&\quad \text{ if } 0\leq -r-s+2\leq n \\
0& \quad \text{ otherwise}.
\end{cases}\notag
 \end{align}
and
\begin{align}
&\psi(t^r\sqrt{P}\partial, t^s\sqrt{P}\partial) \\
 &=\sum_{i=0}^n\sum_{j=0}^n\left(r+\frac{i}{2}\right)\left(r+i-2+\frac{j}{2}\right)(r+i-1)b_ib_j\delta_{r+s+i+j-2,0}\omega_0 \notag
\end{align}
For the cross term $\psi(t^r\partial, t^s\sqrt{P}\partial)$, the formula is more complicated and we don't have a simple description of what it is but we have the following in the case $l=0$ (for $l=1$ a similar description can be made).
For all $r,s\in\mathbb Z$, we have
\begin{align}
&\psi(t^r\partial, t^s\sqrt{P}\partial) \label{crossterm}\\
&=\overline{rt^{r-1}\sqrt{P}\partial\left(\sum_{j=0}^n\left(s+\frac{j}{2}\right)b_jt^{s+j-1}\right)} \notag \\
&=\overline{rt^{r-1}\sqrt{P} \left(\sum_{j=0}^n\left(s+\frac{j}{2}\right)(s+j-1)b_jt^{s+j-2}\sqrt{P}\right)} \notag\\
&=r\sum_{i=0}^n\left(\sum_{j=0}^n\left(s+\frac{j}{2}\right)(r+s+j-2)b_j\right)b_i\overline{t^{r+s+i+j-3}}.\notag
   \end{align}

 We now consider the $\overline{t^{s+i+j-2}}$:  As for any $r\in\mathbb Z$ we have
\begin{align*}
\partial(t^r\sqrt{P})
&=\sum_{i=0}^n\left(r+\frac{i}{2}\right)b_it^{r+i-1}.
\end{align*}
so in the quotient $R/\partial R$ we have the recursion relation
\begin{align*}
\sum_{i=0}^n\left(r+\frac{i}{2}\right)b_i\overline{t^{r+i-1}}=0.
\end{align*}
As $b_n=1$, this proves that for $r\geq 0$
\begin{align*}
\overline{t^{r+n-1}}=-\left(r+\frac{n}{2}\right)^{-1}\sum_{i=0}^{n-1}\left(r+\frac{i}{2}\right)b_i\overline{t^{r+i-1}}.
\end{align*}
This implies that $\{\overline{t^{k}}\,|\, k\geq n-1\}$ can be written as a linear combination of $\{\overline{t^{-1}},\dots, \overline{t^{n-2}}\}$.
More precisely we can find polynomials $p_{r,i}=p_{r,i}(b_0,\dots, b_{n-1})$, $-1,\leq i\leq n-2$, such that
\begin{equation}\label{pki}
\overline{t^{k}}=\sum_{i=-1}^{n-2}p_{k,i}\overline{t^i}
\end{equation}
for all $k\geq n-1$.  We will set $p_{r,i}=\delta_{r,i}$ for $-1\leq i,r\leq n-2$ so \eqnref{pki} holds for all $k\geq -1$.
As $b_0\neq 0$ we have
\begin{equation}\label{negrecursion}
\overline{t^{r-1}}=-\frac{1}{rb_0}\sum_{i=1}^n\left(r+\frac{i}{2}\right)b_i\overline{t^{r+i-1}}.
\end{equation}
This means that $\{\overline{t^{-k}}\,|\, k\geq 2\}$ can be written as a linear combination of $\{\overline{t^{-1}},\dots, \overline{t^{n-2}}\}$.
%
More precisely we can find rational functions $q_{r,i}=q_{r,i}(b_0,\dots, b_{n-1})$, $-1,\leq i\leq n-2$, such that
\begin{equation}\label{qki}
\overline{t^{-k}}=\sum_{i=-1}^{n-2}q_{k,i}\overline{t^i}
\end{equation}
for all $k\geq 1$.

As a consequence of \eqnref{crossterm}, \eqnref{pki} and \eqnref{qki}, we have 
\begin{align}
&\psi(t^r\partial, t^s\sqrt{P}\partial) \label{crossterm2}\\
&= r\sum_{k=-1}^{n-2}\left(\sum_{i,j=0}^n\Theta(r+s+i+j-2)\left(\left(s+\frac{j}{2}\right)(r+s+j-2)b_j\right)b_i p_{r+s+i+j-3,k}(b)\right) \overline{t^k} \notag \\
&\quad +r\sum_{k=-1}^{n-2}\left(\sum_{i,j=0}^n\Theta(-(r+s+i+j-1))\left(\left(s+\frac{j}{2}\right)(r+s+j-2)b_j\right)b_iq_{-(r+s+i+j-3),k}(b)\right)\overline{t^k}  \notag 
   \end{align}
   where 
   $$
   \Theta(k)=\begin{cases}1 & \text{ if } k\geq 0 \\ 
   0 & \text{ for } k<0,
   \end{cases}
   $$
   is the Heaviside function.
Because of this we then turn to giving very general formulae for the polynomials $p_{k,i}$ and $q_{k,i}$.

If we let $p_{i,r}=p_{r,i}(b_0,\dots, b_{n-1})$ be polynomials in $b_0,\dots, b_{n-1}$, $i=-1,0,1,2\dots, n-2$, $r\geq -1$ that satisfy the recursion relation
 \begin{equation}\label{recursionrln2}
\sum_{j=0}^n\left(r+\frac{j}{2}\right)p_{r+j-1,i}(b_0,\dots, b_{n-1})
b_j =0\quad\text{ for all }r\geq n-1,
\end{equation}
with initial conditions $p_{r,i}=\delta_{i,r}$ for $-1\leq i,r\leq n-2$.  Set
\begin{equation}
P_i(b,z):=\sum_{k\geq 0}p_{k-1,i}z^k=\sum_{k\geq -1}p_{k,i}z^{k+1},
\end{equation}
and
\begin{gather*}
 \bar P(z):=\sum_{j=0}^nb_jz^{n-j},\quad Q(z):=z \bar P'(z) -n \bar P(z), \\
   R_i(z):=z^{n}\sum_{j=0}^n\sum_{-j\leq l<n-1}(2l+j)b_jp_{l+j-1,i}z^{l},
\end{gather*}
where $-1\leq i\leq n-2$.
Note that $\bar P(z)$, $Q(z)$ and $R_i(z)$ are all polynomials in $z$ with coefficients polynomials in the $b_i$.
Then we have
\begin{align}\label{1storderode}
2z&  \bar P(z)\frac{d}{dz}P_i(b,z)+Q(z)P_i(b,z)
=R_i(z),
\end{align}
for $i=1,\dots,n$.
This has integrating factor
\begin{equation}
\mu=\exp\left(\int \frac{\bar P'(z)}{2  \bar P(z)}-\frac{n}{2z}\,dz\right)=\sqrt{ \bar P(z)}z^{-n/2}.
\end{equation}
Here $z^{-n/2}$ is only to be used formally in the following when integrating and is not to be thought of as a multivalued function.   For $i=-1,\dots,n-2$ we obtain generating functions as solutions to the family of $n$ equations \eqnref{1storderode}:
\begin{equation}\label{slns}
P_i(b,z)=\frac{z^{n/2}}{\sqrt{ \bar P(z)}}\int\frac{R_i(z)}{2z^{\frac{n+2}{2}} \sqrt{\bar P(z)}}\,dz
\end{equation}
and they are given by hyperelliptic integrals.
In the above formula we expand $\displaystyle{\frac{1}{\bar P(z)^{1/2}}}$ as a Laurent series in $z$, multiply by $R_i(z)/2z^{(n+2)/2}$ and then integrate formally term by term.  Let us explain this more precisely:  If we expand $\displaystyle{\frac{1}{\sqrt{\bar P(t)} }=\sum_{k=0}^\infty h_k(b_0,\dots,b_{n-1})t^k}$,
then \eqnref{slns} becomes
\begin{align}\label{slns1.5}
P_i(b,z)
&=\frac{1}{2}\sum_{k=0}^\infty   \sum_{m=0}^\infty \sum_{j=0}^n\sum_{-j\leq l<n-1}h_k(b_0,\dots,b_{n-1})(2l+j)b_jp_{l+j-1,i}  \\
&\hskip 150pt \times \Big(\int  h_m(b_0,\dots,b_{n-1})z^{\frac{n-2}{2}+m+l}\,dz \Big)z^{k+(n/2)}\notag \\
&=\frac{1}{2}\sum_{k=0}^\infty   \sum_{m=0}^\infty \sum_{j=0}^n\sum_{-j\leq l<n-1}(2l+j)b_jp_{l+j-1,i}h_k(b_0,\dots,b_{n-1})h_m(b_0,\dots,b_{n-1}) Cz^{k+(n/2)}\notag    \\
&\quad + \sum_{k=0}^\infty   \sum_{m=0}^\infty \sum_{j=0}^n\sum_{-j\leq l<n-1}\frac{(2l+j)b_jp_{l+j-1,i}}{n +2(m+l)}h_k(b_0,\dots,b_{n-1})h_m(b_0,\dots,b_{n-1})z^{m+n+l + k}.\notag 
\end{align}
where $C$ is a constant of integration (which is zero if $n$ is odd as the left hand power series has no fractional powers of $z$). Then \eqnref{slns1.5} gives us a (rather complicated) description of $p_{a,i}$ as coefficient of $z^a$ in the above power series.
So we are left with providing a description of $h_k(b_0,\dots,b_{n-1})$ (see \eqnref{hkb} below).

One can expand $1/\sqrt{ P(z)}$ using Bell polynomials and Fa\`a di Bruno's formula as follows.
The Bell polynomials in the variables $z_1,z_2,z_3,\dots$ are defined to be
\begin{align*}
B_{m,k}(z_1,\dots,z_{m-k+1}):=\sum \frac{m!}{l_1!l_2!\cdots l_{m-k+1}!}\left(\frac{z_1}{1!}\right)^{l_1} \cdots \left(\frac{z_{m-k+1}}{(m-k+1)!!}\right)^{l_{m-k+1}}
\end{align*}
where the sum is over $l_1+l_2+\cdots =k$ and $l_1+2l_2+3l_3+\cdots =m$.

Now Fa\`a di Bruno's formula for the $m$-derivative of $f(g(x))$ is
\begin{align*}
\frac{d^m}{dx^m}f(g(x))=\sum_{l=0}^mf^{(l)}(g(x))B_{m,l}(g'(x),g''(x),\dots, g^{(m-l+1)}(x)).
\end{align*}
Setting $f(x)=1/\sqrt{x}$, $g(x)=\bar P(x)$ we get
\begin{equation}
f^{(m)}(x)=\frac{(-1)^m(2m-1)!!}{2^mx^{(2m+1)/2}}
\end{equation}
and $\bar P^{(k)}(0)=k!b_{n-k}$ so that
\begin{align*}
\frac{d^m}{dx^m}f(g(x))|_{x=0}=\sum_{l=0}^m\frac{(-1)^ l(2 l-1)!!}{2^ lb_n^{(2l+1)/2}}B_{m, l}(b_{n-1},2b_{n-2},\dots,(m-l+1)!b_{n-m+l-1}),
\end{align*}
where we set $b_k=0$ for $k\leq -1$.
As a consequence
\begin{align*}
\frac{1}{\sqrt{\bar P(t)} }&=\sum_{k=0}^\infty \frac{d^k}{d t^k}f(g( t))|_{ t=0} t^k \\
&\hskip-.6cm =\sum_{k=0}^\infty \frac{1}{k!} \left(\sum_{l=0}^k\frac{(-1)^ l(2 l-1)!!}{2^ lb_n^{(2l+1)/2}}B_{k, l}(b_{n-1},2b_{n-2},\dots,(k- l+1)!b_{n-k+l-1})\right) t^k ,
\end{align*}
and hence 
\begin{equation}\label{hkb}
h_k(b_0,\dots,b_{n-1})= \frac{1}{k!} \left(\sum_{l=0}^k\frac{(-1)^ l(2 l-1)!!}{2^ lb_n^{(2l+1)/2}}B_{k, l}(b_{n-1},2b_{n-2},\dots,(k- l+1)!b_{n-k+l-1})\right).
\end{equation}

 Lastly we derive the formulae for the $q_{k,i}$ in \eqnref{qki} in terms of generating functions.
 For $-1\leq i\leq n-2$ set
 \begin{equation}
 Q_i(b,z):=\sum_{k\geq -n+2}q_{k,i}z^{k+n-2}=\sum_{k\geq 0}q_{k-n+2,i}z^{k}
 \end{equation}
 where $q_{k,i}$, $k\geq -n+2$, are rational functions in $b_i$ and satisfy the recursion relation
 \begin{equation}
\overline{(t^{-1})^{1-r}}=-\frac{1}{rb_0}\sum_{i=1}^n\left(r+\frac{i}{2}\right)b_i\overline{t^{r+i-1}}.
\end{equation}
or
\begin{equation}
\overline{(t^{-1})^{k}}=-\frac{1}{(1-k)b_0}\sum_{i=1}^n\left(1-k+\frac{i}{2}\right)b_i\overline{(t^{-1})^{k-i}}.
\end{equation}
This gives us a recursion relation
 \begin{align}\label{qk2}
\sum_{j=0}^n\left(2k-(j+2) \right)b_jq_{k-j,i}=0.
\end{align}
for all $k\geq 2$.  These polynomials satisfy the initial condition $q_{k,i}=\delta_{k,-i}$ for $-n+2\leq i,k\leq 1$.
Set
\begin{gather}
P(z)=\sum_{i=0}^nb_iz^i,\qquad Q(z):=zP'(z)-2(n-1)P(z) ,\label{pqs}\\
   S_i(z):= \sum_{j=0}^n\sum_{j-n+2\leq l<2}(2l-(j+2))b_jq_{l-j,i}z^{l+n-2} .\notag
\end{gather}
Then
\begin{align*}
2z    P(z)\frac{d}{dz}Q_i(b,z)+Q(z)Q_i(b,z)
&=S_i(z)\end{align*}
which has integrating factor
$$
\mu=\exp\int \frac{Q(z)}{2zP(z)}\,dz=z^{-(n-1)}\sqrt{P(z)}
$$
Thus
\begin{equation}\label{slns2}
Q_i(b,z)=\frac{z^{n-1}}{\sqrt{   P(z)}}\int\frac{S_i(z)}{2z^n \sqrt{  P(z)}}\,dz.
\end{equation}

As in the case of $P_i(b,z)$, one can use F\'aa de Bruno's formula and Bell polynomials to find the Taylor series expansion of $Q_i(b,z)$ at $z=0$.

Next we illustrate the above generating functions in the examples below.
First we can directly obtain

\begin{lemma}
If $P(t)=t^2-2bt$ with $b\ne0$, then the center of the universal central extension of $\mathcal{R}(P)$ is two dimensional and we can write any 2-cocycle $\psi$ as
\begin{align*}
\psi(t^r\partial, t^s\partial)&=\left( r^3\delta_{r+s,0}-br(r-1)(2r-1)\delta_{r+s-1,0}\right)\omega_0,
 \\
\psi(t^r\partial, t^s\sqrt{P}\partial)&=3rb^{r+s+1}\left(r \left(4 s^2+5 s+2\right)+4 s^3+10 s^2+9 s+3\right)\\
&\hskip 100pt
\times \frac{(2r+2s-3)!!}{(r+s+1)!}\overline{1},\\
\psi(t^r\sqrt{P}\partial, t^s\sqrt{P}\partial)&=(r+1)^3\delta_{r+s+2,0}\omega_0 \\
&\quad -b(2r+1)(2r^2+2r+1)\delta_{r+s+1,0}\omega_0\\
&\quad +b^2(4r^2-1)r\delta_{r+s,0}\omega_0.
\end{align*}
where, by definition, $(2r+2s-1)!!=(2r+2s-1)\cdot(2r+2s-5)\cdots 5\cdot3\cdot1$ and
$$
 \frac{(2r+2s-3)!!}{(r+s+1)!}\overline{1}=\begin{cases} 0 &\text{ if $r+s\leq -2$} \\
  \frac{1}{3} &\text{ if $r+s= -1$} \\
-1 &\text{ if $r+s=0 $} .
\end{cases}
$$\end{lemma}

\subsection{The case $R=\mathbb C[t^{\pm1},u\,|\,u^2=t^2-2bt+c]$ with $c\ne0$} Up to isomorphism we may assume that $c=1$.
In the case $P(t)=t^2-2bt+1$, $l=0$, the center of the universal central extension is three dimensional and in terms of the basis elements $\omega_0:=\overline{t^{-1}\sqrt{P}}$,  $\overline{t^{-1}}$, $\overline{1} $,

\begin{align*}
\psi(t^r\partial, t^s\partial)&= \big(r^3\delta_{r+s,0}-br(r-1)(2r-1)\delta_{r+s-1,0} \\
&\hskip 100pt+r(r-1)(r-2)\delta_{r+s-2,0}\big)\omega_0
 \\
\psi(t^r\sqrt{P}\partial, t^s\sqrt{P}\partial)&=(r+1)^3\delta_{r+s+2,0} \omega_0\\
&\quad -b(2r+1)\left(2r^2+2r+1\right)\delta_{r+s+1,0}\omega_0\\
&\quad +\left(b^2 r \left(4 r^2-1\right)+2 r \left(r^2+1\right)\right)\delta_{r+s,0}\omega_0\\
&\quad -2 b (r-1) r (2 r-1)\delta_{r+s-1,0} \omega_0 \\
&\quad +r(r-1)(r-2)\delta_{r+s-2,0} \omega_0.
\end{align*}

In this case
\begin{align*}
\partial(t^k\sqrt{P})
&=(k+1)t^{k+1}-b\left(2k+1\right)t^k+kt^{k-1}
\end{align*}
so in the quotient $R/\partial R$ we have
\begin{align}\label{recursionrln}
(k+1)\overline{t^{k+1}}-b\left(2k+1\right)\overline{t^k}+k\overline{t^{k-1}} =0.
\end{align}
This is the recursion relation for the Legendre polynomials $p_k(b)$.
Thus for $k\geq 0$
\begin{align*}
\overline{t^{k}}&=p_{k}(b)\overline{1},
\end{align*}
For $k\leq  1$ we have
\begin{align*}
\overline{t^{-k}}&=p_{k-1}(b)\overline{t^{-1}}.
\end{align*}

Lastly we calculate
\begin{align*}
\psi(t^r\partial, t^s\sqrt{P}\partial)
&=r \left(s^2-s\right) \overline{t^{r+s-3}} \\
&\quad +r b \left( s-4  s^2\right)\overline{ t^{r+s-2}} \\
&\quad +r \left((4  s^2+2 s)b^2+2 s^2+s+1\right)\overline{ t^{r+s-1}} \\
&\quad -r b \left(4 s^2+5  s+2 \right)\overline{t^{r+s}}\\
&\quad  +r \left(s+1\right)^2 \overline{t^{r+s+1}}.
   \end{align*}
When $r+s\geq 3$ we get
\begin{align*}
&\psi(t^r\partial, t^s\sqrt{P}\partial) \\
&=r \left(s^2-s\right)p_{r+s-3}(b)\overline{1}  +r b \left( s-4  s^2\right)  p_{r+s-2}(b)\overline{1} \\
&\quad +r \left((4  s^2+2 s)b^2+2 s^2+s+1\right)p_{r+s-1}(b)\overline{1} \\
&\quad -r b \left(4 s^2+5  s+2 \right)p_{r+s}(b)\overline{1} +r \left(s+1\right)^2 p_{r+s+1}(b)\overline{1}.
   \end{align*}
If $r+s= 2$ we get
\begin{align*}
&\psi(t^r\partial, t^s\sqrt{P}\partial) \\
&=r (r-1)(r-2)\overline{t^{-1}}
+\frac{1}{2} b r \left(b^2 \left(r^2-3 r+1\right)-3 r^2+9 r-5\right)\overline{1}.
   \end{align*}
   If $r+s=1$ one has
   \begin{align*}
\psi(t^r\partial, t^s\sqrt{P}\partial)
&=-3 b (r-1)^2 r\overline{ t^{-1}}  +\frac{1}{2} r \left(b^2 \left(3 r^2-6 r+2\right)+3 r^2-6 r+4\right)\overline{1}.
   \end{align*}
When $r+s= 0$ we get
\begin{align*}
&\psi(t^r\partial, t^s\sqrt{P}\partial) \\
&=\frac{1}{2} r \left(3r(r-1)b^2+3 r^2-3 r+2\right)\overline{ t^{-1}} -b r \left(3 r^2-3 r+1\right) \overline{1}.
   \end{align*}
If $r+s=-1$ one obtains
\begin{align*}
\psi(t^r\partial, t^s\sqrt{P}\partial)
&=\frac{1}{2} b r \left(b^2 \left(r^2-1\right)-3 r^2+1\right)
\overline{t^{-1}} +r^3 \overline{1}.
   \end{align*}
When $r+s\leq -2$ we get
\begin{align*}
&\psi(t^r\partial, t^s\sqrt{P}\partial) \\
&=r \left(s^2-s\right) p_{-(r+s)+2}(b)\overline{t^{-1}}  +r b \left( s-4  s^2\right)p_{-(r+s)+1}(b) \overline{t^{-1}}\\
&\quad +r \left((4  s^2+2 s)b^2+2 s^2+s+1\right)p_{-(r+s)}(b) \overline{t^{-1}}\\
&\quad -r b \left(4 s^2+5  s+2 \right)p_{-(r+s)-1}(b)\overline{t^{-1}} +r \left(s+1\right)^2 p_{-(r+s)-2}(b)\overline{t^{-1}}.
   \end{align*}

\subsection{The case $R=\mathbb C[t^{\pm1},u\,|\,u^2=t^3+at+b]$ with $b\neq 0$.}
  Up to isomorphism we may assume that $b=1$. The center of the universal central extension is four dimensional and in terms of the basis elements $\omega_0:=\overline{t^{-1}\sqrt{P}}$,  $\overline{t^{-1}}$, $\overline{1} $, $\overline{t} $. We will express $\overline{t^r} $ in terms of this basis.

Here we have $P(t)=t^3+at+1$ and
\begin{align}
\partial(t^k\sqrt{P})
&=\sqrt{P}\partial(t^k\sqrt{P})=kt^{k-1}P+\frac{1}{2}t^kP' \notag \\
&=\left(k+\frac{3}{2}\right)t^{k+2}+\left(k+\frac{1}{2}\right)at^{k}+kbt^{k-1}
\end{align}
so that
\begin{align*}
 \bar{t^{2}}&=-\frac{a}{3}\bar 1\\
\bar{t^{3}}&= -\frac{3}{5}a\bar{t} -\frac{2}{5}\bar{1}  \\
\bar{t^{4}}&=-\frac{4}{7}\bar{t}+ \frac{5}{21} a^2\bar{1}  \\
\bar{t^{5}}&=\frac{7}{15}a^2\bar{t}+\frac{8a}{15}\bar{1} \\
\vdots &\qquad \vdots
\end{align*}
If we write
\begin{equation}
\overline{t^r}=p_{r,1}\overline{t} +p_{r,0}\overline{1}+p_{r,-1}\overline{t^{-1}}
\end{equation}
for some polynomials $p_{r,i}$, $i=-1,0,1$ in $a$
so that $p_{k,i}$, $i=-1,0,1$ satisfy the same recurrence relation as $t^k$ but with initial conditions
\begin{align*}
p_{-1,-1}&=0,\enspace p_{0,-1}=0, \enspace p_{1,-1}=0, \\
p_{-1,0}&=0,\enspace p_{0,0}=1,\enspace p_{1,0}=0, \\
p_{-1,1}&=0,\enspace p_{0,1}=0,\enspace p_{1,1}=1.
\end{align*}
So we get $Q(z)=-3-az^2$,
\begin{align}
R_{-1}(z)&=0\\
R_{0}(z)&=z^3\sum_{j=0}^n\sum_{-j\leq l\leq 0}(2l+j)b_jp_{l+j-1,0}z^{l} =-z,\\
R_{1}(z)&=z^3\sum_{j=0}^3\sum_{-j\leq l\leq 0}(2l+j)b_jp_{l+j-1,1}z^{l}=z^2,
\end{align}
and the generating series are given by \eqnref{slns} or rather \eqnref{slns1.5}
\begin{align*}
P_{0}(z)&=\frac{z^{3/2}}{\sqrt{ 1+az^2+z^3}}\int\frac{-z}{2z^{5/2} \sqrt{1+az^2+z^3}}\,dz \\ \\
&=z-\frac{a
   z^3}{3}-\frac{2  z^4}{5}+\frac{5 a^2 z^5}{21}+\frac{8}{15} a  z^6 +\left(\frac{16}{55}-\frac{15 a^3}{77}\right)z^7+\cdots  \\
P_{1}(z)&=\frac{z^{3/2}}{\sqrt{ 1+az^2+z^3}}\int\frac{z^{2}}{2z^{5/2} \sqrt{1+az^2+z^3}}\,dz \\
&=z^2 -\frac{3 a z^4}{5}-\frac{4
   z^5}{7}+\frac{7 a^2 z^6}{15}+\frac{348}{385} a  z^7+\cdots \\
\end{align*}

Similarly one has from \eqnref{negrecursion}
\begin{align*}
 \overline{t^{-1}}&=\overline{t^{-1}},\\
\overline{t^{-2}}&= -\frac{a}{2}\overline{t^{-1}} +\frac{1}{2}\overline{t} , \\
\overline{t^{-3}}&= \frac{3a^2}{8}\overline{t^{-1}} -\frac{1}{4}\overline{1}-\frac{3a}{8}\overline{t}, \\
\overline{t^{-4}}&= -\frac{15a^3+24b^2}{48}\overline{t^{-1}} +\frac{5a}{24}\overline{1}
+\frac{5a^2}{16}\overline{t}, \\
\vdots &\qquad \vdots  \\
\overline{t^{-k}}&=q_{k,-1}\overline{t^{-1}} +q_{k,0}\overline{1}+q_{k,1}\overline{t}.
\end{align*}
Then one needs to solve the differential equations
\begin{align*}
2z    P(z)\frac{d}{dz}Q_i(b,z)+Q(z)Q_i(b,z)
&=S_i(z), \quad i=-1,0,1.
\end{align*}
From \eqnref{pqs}
\begin{align*}
S_{-1}(z)&=0 \\
S_0(z)&=-2z -a z^2 \\
S_1(z)&=-4-3a z
\end{align*}
so that by \eqnref{slns2} (and F\'aa de Bruno's formula and Bell polynomials) 
\begin{align*}
Q_{-1}(z)
&=\frac{z^{2}}{\sqrt{ z^3+az+1}} \\
&=z^2-\frac{a}{2}z^3+\frac{3a^2}{8}z^4-\frac{15a^3+24}{48}z^5+\cdots   \\
Q_0(z)
&=\frac{z^{2}}{\sqrt{ z^3+az+1}}\int\frac{-2z -a z^2 }{2z^{3} \sqrt{1+az^2+z^3}}\,dz  \\
&=z-\frac{1}{4}z^4+\frac{5a}{24}z^5+\cdots   \\
Q_1(z)
&=\frac{z^{2}}{\sqrt{ z^3+az+1}}\int\frac{-4-3a z }{2z^{3} \sqrt{1+az^2+z^3}}\,dz  \\
&=1+\frac{1}{2}z^3-\frac{3a}{8}z^4+\frac{5a^2}{16}z^5+\cdots
\end{align*}
where one needs to use a constant of integration $1$, $a/2$, and $-a^2/8$ respectively.  Recall that (2.26) looks like the same recursion relation as (2.22) but allows one to solve for the negative powers of $\overline{t^{-k}}$ as a linear combination of $\left\{\overline{t^{-1}},\overline{1},\overline{t}\right\}$.  Thus we needed to use a different set of initial conditions $q_{k,i}=\delta_{k,-i}$, $-1\leq k\leq 1$, $-1\leq i\leq 1$ than for the $p_{k,i}$.  Using \eqnref{pqs} and \eqnref{slns2} one is lead to the calculation given above for $Q_i$, $i=-1,0,1$.

\subsection{The case $R=\mathbb C[t^{\pm1},u\,|\,u^2=t^{2r}-2bt^r+1]$, $r\in\mathbb N$.}   We assume $r\in \mathbb N$ is fixed throughout this subsection. The center of the universal central extension is $2r+1$ dimensional and in terms of the basis elements $\omega_0:=\overline{t^{-1}\sqrt{P}}$,  $\overline{t^{i}}$ for $i=-1, 0, 1, ...2r-2$. We will explain how to express $p_k=\overline{t^k} $ in terms of this basis. In this case we have
\begin{align*}
\partial(t^k\sqrt{P})
&=(k+r)t^{2r+k-1}-b\left(2k+r\right)t^{r+k-1}+kt^{k-1}
\end{align*}
which leads one to the recursion relation
\begin{align*}
(k+r)p_{2r+k-1}-b\left(2k+r\right)p_{r+k-1}+kp_{k-1}=0
\end{align*}

Let now $\alpha,\beta\in\mathbb C$, $c\in\mathbb R$, $c>0$ and $\gamma:=\alpha+\beta+1$. The associated Jacobi polynomials $\rho_k=P_k^{(\alpha,\beta)}(x;c)$, $k\in \mathbb Z_+$,  $\rho_{-1}=0$, $\rho_0=1$  satisfy the recursion relation
\begin{align*}
2(k+c+1)(&k+c+\gamma)(2k+2c+\gamma-1)\rho_{k+1} \\
&=(2k+2c+\gamma)\Big((2k+2c+\gamma-1)(2k+2c+\gamma +1)x \\
&\quad +(\gamma-1)(\gamma -2\beta-1)\Big)\rho_k-2(k+c+\gamma-\beta-1) \\
&\quad \times (k+c+\beta)(2k+2c+\gamma+1)\rho_{k-1}
\end{align*}
for all $k\in\mathbb Z_+$.   Setting $c=1$ gives the usual recursion relation for the Jacobi polynomials. Setting $\alpha=\beta=0$  into the associated Jacobi recursion relation above gives us
\begin{align*}
 (k+c+1) \rho_{k+1} =(2k+2c+1) x\rho_k-(k+c ) \rho_{k-1}
\end{align*}
for $k\geq 0$.
The recursion relation we have is
\begin{align*}
(k+r)p_{2r+k-1}-b\left(2k+r\right)p_{r+k-1}+kp_{k-1}=0
\end{align*}
Let $k=rs+q$ where $0\leq q\leq r-1$.   Then this
becomes
\begin{align*}
(rs+q+r)p_{2r+rs+q-1}-b\left(2(rs+q)+r\right)p_{r+rs+q-1}+(rs+q)p_{rs+q-1}=0
\end{align*}
For $l\in\mathbb Z$ with $r$ and $0\leq q\leq r-1$ setting $P_l^{q/r}(b):=p_{(l+1)r+q-1}$ we get
\begin{align*}
\left(s+\frac{q}{r}+1\right)P_{s+1}^{q/r}(b)-b\left(2s+2\left(\frac{q}{r}\right)+1\right)P_s^{q/r}(b)+\left(s+\frac{q}{r}\right)P_{s-1}^{q/r}(b)=0.
\end{align*}
Thus the $P_l^{q/r}(b)$ and hence the $p_{(l+1)r+q-1}$ are associated Legendre polynomials (or functions as the case may be).  Similarly the $q_{(l+1)r+q-1}$ are associated Legendre polynomials but with a different set of initial conditions. 
Then one can rewrite (\ref{crossterm2}) in terms of the associate Legendre polynomials and the basis $\{\overline{t^{-1}},\dots, \overline{t^{n-2}}\}$.

\section{Chebyshev polynomials and some related polynomials}

We recall below our description of the group of units $R^*(P)$
for the interesting case of $P(t)$ studied by
Date, Jimbo, Kashiwara and Miwa \cite{DJKM} where they investigated
integrable systems arising from Landau-Lifshitz differential
equation. Let $$\displaystyle{P(t)=\frac{t^4-2\beta
t^2+1}{\beta^2-1}},\ \ \beta\neq \pm 1.$$ Observe that in this case
$P(t)=q(t)^2-1$ where $q(t)=\frac{t^2-\beta}{\sqrt{\beta^2-1}}$.

For convenience, let
\begin{equation*}\begin{split}
& \lambda_0=\frac{t^2-\beta}{\sqrt{\beta^2-1}}+\sqrt{P},\\
& \lambda_1=\frac{t^2+1}{\sqrt{2(\beta+1)}}+\sqrt{\frac{\beta-1}{2}}\sqrt{P},\\
& \lambda_2=\frac{t^2- 1}{\sqrt{2(\beta-1)}}+\sqrt{\frac{\beta+1}{2}}\sqrt{P}.\\
\end{split}\end{equation*}
It is easy to verify that $\lambda_0, \lambda_1, \lambda_2\in R^*(P)$. Actually,
$\lambda_0 \bar\lambda_0=1$, $\lambda_1\bar\lambda_1=t^2$,
$\lambda_2\bar\lambda_2=t^2$, $\lambda_1\lambda_2=t^2\lambda_0$.
We have

\begin{theorem} \cite[Theorem 13]{CGLZ2}\label{CGLZ2}
Let $P(t)$ be as above.
\begin{itemize}\item[(a).]  As a multiplicative group, $R^*(P)$ is generated by $\C^*$, $t,
 \lambda_1, \lambda_2$.
\item[(b).] $R^*(P)\simeq\mathbb C^*\times\Z\times\Z\times\Z.$
\end{itemize}
\end{theorem}

Next we will study properties of elements in the unit group $R^*(P)$.
 \begin{theorem} Define polynomials $u_n,v_{n-1} $ by $u_n+v_{n-1}\sqrt{P}=\lambda_0^n$ for $n\in\mathbb Z$. Then
 $$u_{n+2}(t)-2q(t)u_{n+1}(t)+u_n(t)=0, \  v_{n+2}(t)-2q(t)v_{n+1}(t)+v_n(t)=0,
 $$
 and
 \begin{equation}
 -P\frac{dq}{dx}\frac{d^2y}{dx^2}+ \left(P \frac{d^2q}{dx^2} -q\left(\frac{dq}{dx}\right)\right)\frac{dy}{dx}+n^2\left(\frac{dq}{dx}\right)^2y=0,
\end{equation}
 is the second order differential equation satisfied by the $u_n$. 

The $v_n(t)$ satisfies $$-P\frac{dq}{dx}\frac{d^2y}{dx^2}+ \left(P \frac{d^2q}{dx^2} - 3q\left(\frac{dq}{dx}\right)\right)\frac{dy}{dx}+n(n+2)\left(\frac{dq}{dx}\right)^2y=0.$$
 \end{theorem}

\begin{proof} We have
\begin{align*}
\sum_{n\geq 0}u_n(t)z^n+ \sum_{n\geq 0}v_{n-1}(t)\sqrt{P(t)}z^n&= \sum_{n\geq 0} (q(t)+\sqrt{P})^nz^n  \\
&= \frac{1}{1- (q(t)+\sqrt{P(t)})z}\\
&= \frac{1-q(t)z+\sqrt{P(t)}z}{1- 2q(t)z+(q(t)^2- P(t))z^2}\\
&= \frac{1-q(t)z+\sqrt{P(t)}z}{1- 2q(t)z+z^2}\\
\end{align*}
We have the generating series for the set of polynomials as
$$
\sum_{n\geq 0}u_n(t)z^n=\frac{1-q(t)z}{1-2q(t)z+z^2},\quad \sum_{n\geq 0}v_{n-1}(t)z^n=\frac{1}{1-2q(t)z+z^2}
$$
The first equation gives us
$$
\sum_{n\geq 0}u_n(t)z^n-\sum_{n\geq 0}2q(t)u_n(t)z^{n+1}+\sum_{n\geq 0}u_n(t)z^{n+2}=1-q(t)z,
$$
or
$$
\sum_{n\geq 2 }u_n(t)z^n-\sum_{n\geq 1}2q(t)u_n(t)z^{n+1}+\sum_{n\geq 0}u_n(t)z^{n+2}=0,
$$
or
$$
\sum_{n\geq 0 }u_{n+2}(t)t^{n+2}-\sum_{n\geq 0}2q(t)u_{n+1}(t)z^{n+2}+\sum_{n\geq 0}u_n(t)z^{n+2}=0
$$
which in turn gives us the recurrence relation
\begin{equation} \label{uns}
u_{n+2}(t)-2q(t)u_{n+1}(t)+u_n(t)=0,
\end{equation}
with $u_0(t)=1$ and $u_1(t)=q(t)$.
The second equation gives us
$$
\sum_{n\geq 0}v_{n-1}(t)z^n-\sum_{n\geq 0}2q(t)v_{n-1}(t)z^{n+1}+\sum_{n\geq 0}v_{n-1}(t)z^{n+2}=1,
$$
or
$$
\sum_{n\geq 2 }v_{n-1}(t)z^n-\sum_{n\geq 1}2q(t)v_{n-1}(t)z^{n+1}+\sum_{n\geq 0}v_{n-1}(t)z^{n+2}=0,
$$
or
$$
\sum_{n\geq 0 }v_{n+1}(t)z^{n+2}-\sum_{n\geq 0}2q(t)v_{n}(t)z^{n+2}+\sum_{n\geq 0}v_{n-1}(t)z^{n+2}=0
$$
which in turn gives us the recurrence relation
\begin{equation} 
v_{n+2}(t)-2q(t)v_{n+1}(t)+v_n(t)=0,
\end{equation} 
with $v_0(t)=1$ and $v_1(t)=2q(t)$.

The Chebyshev polynomials $T_n(x)$ satisfy the recursion relation 
$$
T_{n+2}(x)-2xT_{n+1}(x)+T_n(t)=0,
$$
with $T_0(x)=1$ and $T_1(x)=x$.  We note that this first family of Chebyshev polynomials satisfy the differential equation
$$
(1-t^2)y''(t)-ty'(t)+n^2y=0.
$$
Since this is similar to \eqnref{uns} we use the above differential equation to find one for the $v_n$ and $u_n$ as follows
Now for $q=q(x)$ by the chain rule
\begin{align*}
\frac{d}{dx}y(q)&=\frac{dy(q)}{dq}\frac{dq}{dx},\\
 \frac{d^2}{dx^2}y(q)&=\frac{d}{dq}\left(\frac{dy(q)}{dq}\frac{dq}{dx}\right)\frac{dq}{dx}  \\
 &= \frac{d^2y(q)}{d^2q}\frac{dq}{dx}\frac{dq}{dx}+\frac{dy(q)}{dq}\frac{d}{dq}\left(\frac{dq}{dx}\right)\frac{dq}{dx}   \\
 &= \frac{d^2y(q)}{d^2q}\left(\frac{dq}{dx}\right)^2+\frac{dy(q)}{dq}\frac{d^2q}{dx^2} .
\end{align*}
so that
\begin{align*}
\frac{dy(q)}{dq}&=\frac{1}{\frac{dq}{dx}}\frac{d}{dx}y(q),\\
 \frac{d^2y(q)}{d^2q}
 &= \frac{1}{\left(\frac{dq}{dx}\right)^2}\left(\frac{d^2}{dx^2}y(q)-\frac{dy(q)}{dq}\frac{d^2q}{dx^2} \right) .
\end{align*}
Inserting these into the Chebyshev equation with respect to $q$, we get
\begin{align*}
\frac{1-q^2}{\left(\frac{dq}{dx}\right)^2}\left(\frac{d^2}{dx^2}y(q)-\frac{dy(q)}{dq}\frac{d^2q}{dx^2} \right)-\frac{q}{\frac{dq}{dx}}\frac{d}{dx}y(q)+n^2y(q)=0.
\end{align*}
or
\begin{align*}
(1-q^2)\left(\frac{d^2}{dx^2}y(q)-\frac{dy(q)}{dq}\frac{d^2q}{dx^2} \right)-q\frac{dq}{dx}\frac{d}{dx}y(q)+n^2y(q)\left(\frac{dq}{dx}\right)^2=0.
\end{align*}
or
\begin{equation}
 -P\frac{dq}{dx}\frac{d^2y}{dx^2}+ \left(P \frac{d^2q}{dx^2} -q\left(\frac{dq}{dx}\right)\right)\frac{dy}{dx}+n^2\left(\frac{dq}{dx}\right)^2y=0.
\end{equation}
This is the second order differential equation satisfied by the $u_n$. 

The second family of Chebyshev polynomials satisfy the differential equation
\begin{align*}
(1-t^2)y''(t)-3ty'(t)+n(n+2)y=0.
\end{align*}
so that the polynomials $v_n$ satisfy
\begin{equation}
-P\frac{dq}{dx}\frac{d^2y}{dx^2}+ \left(P \frac{d^2q}{dx^2} - 3q\left(\frac{dq}{dx}\right)\right)\frac{dy}{dx}+n(n+2)\left(\frac{dq}{dx}\right)^2y=0.
\end{equation}\end{proof}

 \begin{theorem} Define polynomials $a_n,b_{n-1} $ by $a_n+b_{n-1}\sqrt{P}=\lambda_2^n$ for $n\in\mathbb N$. Then
 $a_n=b_0^{-1}(b_n-a_1b_{n-1}),$
 and $b_n(t)$ satisfies $$\begin{aligned}
&0 = t^2(t^2-a_1^2)(a_1't-a_1)y''      \\
&\quad +\Big(  -   a_1''t^3 (t^2-a_1^2) + \left(2(1-n)t^3- 3a_1a_1't^2+a_1^2(2n+1)t \right)(a_1't-a_1)\Big)y'
 \notag \\
 &\quad +\Big(- a_1''t^2\left(-nt^2- a_1a_1'(n+2 )t+2a_1^2(n+1) \right)     \notag \\
& +\Big(\left(n(n-1)+ \left(n(a_1')^2-a_1a_1''\right)(n+2 )\right)t^2 - n a_1a_1'
 \left( 2n+1\right)t\Big)(a_1't-a_1)\Big)y.\notag
 \end{aligned}$$
 \end{theorem}

\begin{proof}
We now will apply the technique used in Arfken and Weber's book ``Mathematical Methods for Physicists".
First we have
\begin{align*}
\sum_{n\geq 0}a_nz^n+\sum_{n\geq 1}b_{n-1}z^n\sqrt{P}&=\sum_{n\geq 0}(a_1+b_0\sqrt{P})^nz^n  \\
&=\frac{1}{1-(a_1+b_0\sqrt{P})z}=\frac{1-a_1z+b_0\sqrt{p}z}{1-2a_1z+t^2z^2}
\end{align*}
as $a_1^2-b_0^2P=t^2$.
 The first five polynomials $b_n$ are
 \begin{gather*}
b_0=\sqrt{\frac{\beta+1}{2}},\quad b_1=2a_1b_0 , \\
  b_2=  b_0(4a_1^2-t^2),\quad
 b_3=4a_1b_0(2a_1^2-t^2) \\
b_4=b_0 \left(16 a_1^4-12 a_1^2 t^2+t^4\right).
 \end{gather*}

We differentiate the generating function
 $$
\sum_{n\geq 0} w_n(t)z^{n}= \frac{b_0}{(1-2a_1z+t^2z^2)^\alpha},
 $$
 and then in the end set $\alpha =1$ to get our desired differential equation.   If one takes $\alpha=1$ first and then differentiate, then it seems to us that it is not possible to get the second order differential equation that the $b_n$ satisfy.  However to prove that it is the desired differential equation we need to check by induction using the recursion relation that the differential equation is satisfied.

To that end we calculate

\begin{align*}
\sum_{n\geq 0}n w_n(t)z^{n-1}&=\frac{\alpha b_0(2a_1-2t^2z)}{(1-2a_1z+t^2z^2)^{\alpha+1}}  \\
&=\frac{\alpha b_0(2a_1-2t^2z)}{(1-2a_1z+t^2z^2)}\sum_{n\geq 0} w_n(t)z^n  \\
\end{align*}
so that
\begin{align*}
(1-2a_1z+t^2z^2)\sum_{n\geq 0}n w_nz^{n-1}&=\frac{\alpha b_0(2a_1-2t^2z)}{(1-2a_1z+t^2z^2)^{\alpha}}  \\
&=\alpha (2a_1-2t^2z) \sum_{n\geq 0} w_nz^n  \\
\end{align*}
The left hand side expands out to
\begin{align*}
(1-2a_1z&+t^2z^2)\sum_{n\geq 0}n w_nz^{n-1}\\ &=\sum_{n\geq 0}n w_nz^{n-1}-\sum_{n\geq 0}2na_1 w_nz^{n}+\sum_{n\geq 0}nt^2 w_nz^{n+1} \\
&=w_1+\sum_{n\geq 2}n w_nz^{n-1}
-\sum_{n\geq 1}2na_1 w_nz^{n}+\sum_{n\geq 0}nt^2 w_nz^{n+1}  \\
&=w_1 \\
&\quad +\sum_{n\geq 0}\left((n+2)w_{n+2}
-2(n+1)a_1w_{n+1}+ nt^2 w_n\right)z^{n+1} .
\end{align*}
The right hand side expands out to
\begin{align*}
\alpha (2a_1-2t^2z) \sum_{n\geq 0} w_nz^n&= \sum_{n\geq 0}2\alpha a_1 w_nz^n  - \sum_{n\geq 0}2\alpha t^2  w_nz^{n+1}  \\
&=2\alpha a_1w_0+ \sum_{n\geq 1}2\alpha a_1 w_nz^n  - \sum_{n\geq 0}2\alpha t^2  w_nz^{n+1}  \\
&=2\alpha a_1w_0+ \sum_{n\geq 0}2\alpha \left( a_1 w_{n+1}-t^2 w_n\right)z^{n+1} .
\end{align*}
Equating these two expansions we get
\begin{align*}
2\alpha a_1w_0&=w_1,
\end{align*}
and
\begin{align*}
(n+2)w_{n+2}
-2(n+1)a_1w_{n+1}+ nt^2 w_n=2\alpha \left(a_1w_{n+1}-t^2 w_n\right)
\end{align*}
or
\begin{align*}
(n+2)w_{n+2}
-2a_1(n+1+\alpha )w_{n+1}+ (n+2\alpha )t^2 w_n=0.
\end{align*}
We rewrite this recursion relation as
\begin{align}\label{recursionreln}
 (n+1)w_{n+1}
-2a_1(n+\alpha  )w_{n}+ (n-1+2\alpha )t^2w_{n-1}=0,
 \end{align}
 and then differentiate to get
 \begin{align*}
(n+1)&w_{n+1}'
-2a_1(n+\alpha )w_{n}'-2a_1'(n+\alpha  )w_{n}  \\
&\qquad + (n-1+2\alpha )t^2w_{n-1}'+ 2(n-1+2\alpha)tw_{n-1}=0,
 \end{align*}
 which if we multiply by $t$ and substitute in the recursion relation gives us
  \begin{gather}
\begin{split} 0
&=(n+1)tw_{n+1}'
-2a_1(n+\alpha )tw_{n}'-2(n+\alpha  )(a_1't-2a_1)w_{n}  \\
&\qquad + (n-1+2\alpha )t^3w_{n-1}'-2(n+1)w_{n+1}.\label{recursionrelndir}\end{split}
 \end{gather}

  On the other hand if we differentiate the generating function with respect to $t$ we get
 \begin{align*}
\frac{d}{dt} \frac{b_0}{(1-2a_1z+t^2z^2)^\alpha}&= \frac{\alpha b_0(2a_1'z-2tz^2)}{(1-2a_1z+t^2z^2)^{\alpha+1}} \\
&=\frac{\alpha b_0(2a_1'z-2tz^2)}{(1-2a_1z+t^2z^2)}\sum_{n\geq 0} w_nz^n
 \end{align*}
 So we get
 \begin{align*}
 (1-2a_1z+t^2z^2)\sum_{n\geq 0} w_n'z^n&=\alpha (2a_1'z-2tz^2)\sum_{n\geq 0 } w_nz^{n}
 \end{align*}
 and the left hand side expands to
  \begin{align*}
 (1-2a_1z+t^2z^2)\sum_{n\geq 0} w_n'z^n&= \sum_{n\geq 0} w_n'z^n-\sum_{n\geq 0}2a_1 w_n'z^{n+1}+\sum_{n\geq 0}t^2 w_n'z^{n+2}  \\
&=w'_0+w_1'z-2a_1w'_0z \\
&\quad + \sum_{n\geq 2} w_n'z^n-\sum_{n\geq 0}2a_1 w_n'z^{n+1}+\sum_{n\geq 0}t^2 w_n'z^{n+2}  \\
&=w'_0+w_1'z-2a_1w'_0z \\
&\quad + \sum_{n\geq 0}\left(w_{n+2}'- 2a_1w_{n+1}' + t^2 w_n'\right)z^{n+2}
 \end{align*}
 while the right hand side expands to
 \begin{align*}
\alpha (2a_1'z-2tz^2)\sum_{n\geq 0 } w_nz^{n}&=\sum_{n\geq 0 }2\alpha   a_1' w_nz^{n+1}- \sum_{n\geq 0 }2\alpha  t w_nz^{n+2} \\
&=2\alpha  a_1'w_0z+\sum_{n\geq 1 }2\alpha  a_1' w_nz^{n+1}- \sum_{n\geq 0 }2\alpha  t w_nz^{n+2} \\
&=2\alpha  a_1'w_0z+\sum_{n\geq 0 }2\alpha  \left(a_1'w_{n+1}- tw_n\right)z^{n+2}
 \end{align*}
 so that
 \begin{align*}
2\alpha  \left(a_1'w_{n+1}- tw_n\right)=w_{n+2}' - 2a_1w_{n+1}' +t^2 w_n'.
 \end{align*}
We rewrite as
  \begin{align}\label{firsteqn}
2\alpha  \left(a_1'w_{n}- tw_{n-1}\right)=w_{n+1}' - 2a_1w_{n}' +t^2 w_{n-1}'.
 \end{align}
Which we multiply by $(n+\alpha)$
  \begin{align*}
2\alpha &(n+\alpha) \left(a_1'w_{n}- tw_{n-1}\right)-(n+\alpha)w_{n+1}'  \\
 &\quad+ 2(n+\alpha)a_1w_{n}' -t^2 (n+\alpha)w_{n-1}'=0
 \end{align*} and add to
 \begin{align*}
(n+1)&w_{n+1}'
-2a_1(n+\alpha   )w_{n}'-2a_1'(n+\alpha   )w_{n}  \\
&\qquad + (n-1+2\alpha  )t^2w_{n-1}'+ 2(n-1+2\alpha  )tw_{n-1}=0,
 \end{align*}
 to get
   \begin{align*}
(1-\alpha)&w_{n+1}'+2(\alpha-1)a_1'(n+\alpha   )w_{n}+2(1-n-\alpha) ( \alpha-1)  tw_{n-1}  \\
 &\quad+t^2 (\alpha-1)w_{n-1}'=0
 \end{align*}
 Assuming $\alpha\neq 1$ we get
  \begin{align*}
 &-w_{n+1}'+2a_1'(n+\alpha   )w_{n}+2(1-n-\alpha)   tw_{n-1}+t^2 w_{n-1}'=0
 \end{align*}
 or
 \begin{equation}\label{neweqn}
  w_{n+1}'-t^2 w_{n-1}'=2a_1'(n+\alpha   )w_{n}+2(1-n-\alpha)   tw_{n-1}
 \end{equation}

Now recall the recursion relation
\begin{align}\label{recursionreln2}
(n+1)w_{n+1}
-2a_1(n+\alpha  )w_{n}+ (n-1+2\alpha )t^2w_{n-1}=0,
 \end{align}
We multiply \eqnref{neweqn} by $(n-1+2\alpha )t$ to get
   \begin{align*}
 0&= (n-1+2\alpha )t w_{n+1}'-(n-1+2\alpha )t^3 w_{n-1}'-2a_1'(n+\alpha   )(n-1+2\alpha )tw_{n} \\
 &\quad -2(n-1+2\alpha )(1-n-\alpha)   t^2w_{n-1}
  \end{align*}
 and add to it $2(1-n-\alpha)$ times\eqnref{recursionreln2}
    \begin{align}
0&= (n-1+2\alpha )t w_{n+1}'-(n-1+2\alpha )t^3 w_{n-1}'-2a_1'(n+\alpha   )(n-1+2\alpha )tw_{n}\notag \\
 &\quad -2(n-1+2\alpha )(1-n-\alpha)   t^2w_{n-1}+2(1-n-\alpha)(n+1)w_{n+1} \notag  \\
 &\quad
-4(1-n-\alpha)a_1(n+\alpha  )w_{n}+ 2(1-n-\alpha)(n-1+2\alpha )t^2w_{n-1}\notag \\ \notag\\
&= (n-1+2\alpha )t w_{n+1}'+(n-1+2\alpha )t^3 w_{n-1}'-2a_1'(n+\alpha   )(n-1+2\alpha )tw_{n} \notag\\
 &+2(1-n-\alpha)(n+1)w_{n+1}
-4(1-n-\alpha)a_1(n+\alpha  )w_{n}\notag\\ \notag\\
&= (n-1+2\alpha )tw_{n+1}'-(n-1+2\alpha )t^3 w_{n-1}' +2(1-n-\alpha)(n+1)w_{n+1}  \notag\\
&\quad -2(n+\alpha   )\left(a_1'(n-1+2\alpha )t+2a_1(1-n-\alpha)\right)w_{n} \label{non-1}
\end{align}
to which we add\eqnref{recursionrelndir}
\begin{align}
0
&= (n-1+2\alpha )tw_{n+1}'-(n-1+2\alpha )t^3 w_{n-1}' +2(1-n-\alpha)(n+1)w_{n+1}  \notag\\
&\quad -2(n+\alpha   )\left(a_1'(n-1+2\alpha )t+2a_1(1-n-\alpha)\right)w_{n}\notag  \\
&\quad +(n+1)tw_{n+1}'
-2a_1(n+\alpha )tw_{n}'-2(n+\alpha  )(a_1't-2a_1)w_{n}\notag  \\
&\quad +(n-1+2\alpha )t^3w_{n-1}'-2(n+1)w_{n+1} \notag\\ \notag\\
&= 2(\alpha +n)tw_{n+1}'   -2(n+\alpha)(n+1)w_{n+1}\notag\\
&\quad -2(n+\alpha   )\left(a_1'(n+2\alpha )t-2a_1(n+\alpha)\right)w_{n}\notag  \\
&\quad
-2a_1(n+\alpha )tw_{n}' \notag
 \end{align}
 Thus if $\alpha\neq -n$, one has
 \begin{gather}
 \begin{split}0
&= tw_{n+1}'  -a_1 tw_{n}' - (n+1)w_{n+1}\\
&\quad - \left(a_1'(n+2\alpha )t-2a_1(n+\alpha)\right)w_{n} \label{nn+1}\end{split}
 \end{gather}

Reconsider\eqnref{recursionrelndir} from which we subtract $(n+1)t$ times\eqnref{neweqn}
 \begin{align}
0
&=(n+1)tw_{n+1}'
-2a_1(n+\alpha )tw_{n}'-2(n+\alpha  )(a_1't-2a_1)w_{n}  \notag \\
&\quad + (n-1+2\alpha )t^3w_{n-1}'-2(n+1)w_{n+1} \notag \\
&\quad-(n+1)tw_{n+1}'+(n+1)t^3w_{n-1}'+2a_1'(n+\alpha   )(n+1)tw_{n} \notag\\
&\quad +2(1-n-\alpha)  (n+1)t^2w_{n-1}\notag  \\ \notag \\
&=
-2a_1(n+\alpha )tw_{n}' \notag \\
&\quad + (n-1+2\alpha )t^3w_{n-1}'+(n+1)t^3w_{n-1}'\notag \\
&\quad +2a_1'(n+\alpha   )(n+1)tw_{n} -2(n+\alpha  )(a_1't-2a_1)w_{n} \notag\\
&\quad -2(n+1)w_{n+1} +2(1-n-\alpha)  (n+1)t^2w_{n-1}\notag  \\ \notag \\
&=
-2a_1(n+\alpha )tw_{n}' \notag \\
&\quad + 2(n+\alpha )t^3w_{n-1}'\notag \\
&\quad +2(n+\alpha   )(a_1'(n+1)t  -(a_1't-2a_1))w_{n} \notag\\
&\quad -2(n+1)w_{n+1} +2(1-n-\alpha)  (n+1)t^2w_{n-1}\notag
  \end{align}
We then use\eqnref{recursionreln} to eliminate $w_{n+1}$:
 \begin{align}
0
&=
-2a_1(n+\alpha )tw_{n}' + 2(n+\alpha )t^3w_{n-1}'\notag \\
&\quad +2(n+\alpha   )(a_1'(n+1)t  -(a_1't-2a_1))w_{n} \notag\\
&\quad -4a_1(n+\alpha  )w_{n}+2(n-1+2\alpha )t^2w_{n-1} +2(1-n-\alpha)  (n+1)t^2w_{n-1}\notag  \\ \notag \\
&=
-2a_1(n+\alpha )tw_{n}' + 2(n+\alpha )t^3w_{n-1}'\notag \\
&\quad +2(n+\alpha   )(a_1'(n+1)t  -(a_1't-2a_1))w_{n}-4a_1(n+\alpha  )w_{n} \notag\\
&\quad +2\left((n-1+2\alpha ) +(1-n-\alpha)  (n+1)\right)t^2w_{n-1}\notag  \\ \notag \\
&=
-2a_1(n+\alpha )tw_{n}' + 2(n+\alpha )t^3w_{n-1}'\notag \\
&\quad +2(n+\alpha   )a_1'nt w_{n}\notag\\
&\quad -2(n+\alpha)(n-1)t^2w_{n-1}\notag
  \end{align}
 so that if $\alpha\neq n$ we get
 \begin{align}
0
&=
-2a_1 tw_{n}' + 2 t^3w_{n-1}' +2 a_1'nt w_{n}  -2(n-1)t^2w_{n-1}  .\notag
  \end{align}
Replacing $n$ by $n+1$ in the above we get
 \begin{equation} 
0
=
-a_1 tw_{n+1}' +  t^3w_{n}' + a_1'(n+1)t w_{n+1}  -nt^2w_{n}  \label{2ndeqn}
  \end{equation}

  We add to \eqnref{2ndeqn} to $a_1$ times\eqnref{nn+1}
 \begin{align}
0&=-a_1 tw_{n+1}' +  t^3w_{n}' + a_1'(n+1)t w_{n+1}  -nt^2w_{n} \notag\\
&  +a_1tw_{n+1}'  -a_1^2 tw_{n}' - (n+1)a_1w_{n+1} - a_1\left(a_1'(n+2\alpha )t-2a_1(n+\alpha)\right)w_{n}  \notag \\ \notag \\
&=t^3w_{n}' + a_1'(n+1)t w_{n+1}  -nt^2w_{n}\notag \\
&\quad -a_1^2 tw_{n}' - (n+1)a_1w_{n+1} - a_1\left(a_1'(n+2\alpha )t-2a_1(n+\alpha)\right)w_{n}  \notag \\ \notag \\
 &= t(t^2-a_1^2)w_{n}' + (n+1)(a_1't -a_1)w_{n+1}  \notag \\
&\quad +\left(-nt^2- a_1a_1'(n+2\alpha )t+2a_1^2(n+\alpha)\right)w_{n}  \label{onlywn+1}
\end{align}

If we differentiate this we get
 \begin{align*}
0
 &= t(t^2-a_1^2)w_{n}'' +(3t^2-a_1^2-2ta_1a_1')w_{n}' \\
 &\quad + (n+1)a_1''tw_{n+1}  + (n+1)(a_1't -a_1)w_{n+1}'   \\
&\quad +(-2nt- (a_1')^2(n+2\alpha )t- a_1a_1''(n+2\alpha )t- a_1a_1'(n+2\alpha )  \\
&\quad +4a_1a_1'(n+\alpha))w_{n}+\left(-nt^2- a_1a_1'(n+2\alpha )t+2a_1^2(n+\alpha)\right)w_{n}'  \notag \\ \notag \\
 &= t(t^2-a_1^2)w_{n}''   \\
 &\quad + (n+1)a_1''tw_{n+1}  + (n+1)(a_1't -a_1)w_{n+1}'   \\
&\quad +\left(\left(-2n- \left((a_1')^2+a_1a_1''\right)(n+2\alpha )\right)t+a_1a_1'(3n+2\alpha ))\right)w_{n}  \\
&\quad +\left((3-n)t^2- a_1a_1'(n+2\alpha+2 )t+a_1^2(2n+2\alpha-1)\right)w_{n}'
\end{align*}

The next two steps are to eliminate the term with $w_{n+1}'$ in it and then
use\eqnref{onlywn+1} to remove the resulting equation with an $w_{n+1}$ residing in it.
To that end we multiply the above by $t$ and use\eqnref{2ndeqn} and $a'_1$ times\eqnref{nn+1} to remove the $w_{n+1}'$:
 \begin{align*}
0
 &= t^2(t^2-a_1^2)w_{n}''   \\
 &\quad + (n+1)a_1''t^2w_{n+1}  + (n+1)a_1't^2w_{n+1}'-(n+1)a_1tw_{n+1}'     \\
&\quad +\left(\left(-2n- \left((a_1')^2+a_1a_1''\right)(n+2\alpha )\right)t+a_1a_1'(3n+2\alpha ))\right)tw_{n}  \\
&\quad +\left((3-n)t^2- a_1a_1'(n+2\alpha+2 )t+a_1^2(2n+2\alpha-1)\right)tw_{n}'    \notag \\ \notag \\
 &= t^2(t^2-a_1^2)w_{n}''   \\
 &\quad + (n+1)a_1''t^2w_{n+1}   \\
 &\quad + (n+1)a_1't \left(a_1 tw_{n}' +(n+1)w_{n+1}+ \left(a_1'(n+2\alpha )t-2a_1(n+\alpha)\right)w_{n} \right)\\
 &\quad -(n+1)\left( t^3w_{n}' + a_1'(n+1)t w_{n+1}  -nt^2w_{n} \right)   \\
&\quad +\left(\left(-2n- \left((a_1')^2+a_1a_1''\right)(n+2\alpha )\right)t+a_1a_1'(3n+2\alpha ))\right)tw_{n}  \\
&\quad +\left((3-n)t^2- a_1a_1'(n+2\alpha+2 )t+a_1^2(2n+2\alpha-1)\right)tw_{n}'    \notag \\ \notag \\
 &= t^2(t^2-a_1^2)w_{n}''   \\
 &\quad + (n+1)\Big(a_1''t^2w_{n+1}   \\
 &\qquad +a_1 a_1't^2w_{n}' + (n+1)a_1'tw_{n+1}+a_1'\left(a_1'(n+2\alpha )t-2a_1(n+\alpha)\right)tw_{n}  \\
 &\qquad -t^3w_{n}' - a_1'(n+1)t w_{n+1}  +nt^2w_{n}  \Big)  \\
&\quad +\left(\left(-2n- \left((a_1')^2+a_1a_1''\right)(n+2\alpha )\right)t+a_1a_1'(3n+2\alpha ))\right)tw_{n}  \\
&\quad +\left((3-n)t^2- a_1a_1'(n+2\alpha+2 )t+a_1^2(2n+2\alpha-1)\right)tw_{n}'    \notag \\ \notag \\
 &= t^2(t^2-a_1^2)w_{n}''   \\
 &\quad + (n+1)\Big( a_1''t^2 w_{n+1}   \\
 &\qquad +\left(a_1 a_1't^2-t^3\right)w_{n}' +\left(nt^2+a_1'\left(a_1'(n+2\alpha )t^2-2a_1(n+\alpha)t\right)\right)w_{n}\Big)  \\
&\quad +\left(\left(-2n- \left((a_1')^2+a_1a_1''\right)(n+2\alpha )\right)t+a_1a_1'(3n+2\alpha ))\right)tw_{n}  \\
&\quad +\left((3-n)t^2- a_1a_1'(n+2\alpha+2 )t+a_1^2(2n+2\alpha-1)\right)tw_{n}'    \notag \\ \notag \\
 &= t^2(t^2-a_1^2)w_{n}''   \\
 &\quad + (n+1)a_1''t^2  w_{n+1}   \\
 &\qquad + (n+1)\left(a_1 a_1't^2-t^3\right)w_{n}' + (n+1)
 \left(nt^2+a_1'\left(a_1'(n+2\alpha )t^2-2a_1(n+\alpha)t\right)\right)w_{n}  \\
&\quad +\left(\left(-2n- \left((a_1')^2+a_1a_1''\right)(n+2\alpha )\right)t+a_1a_1'(3n+2\alpha ))\right)tw_{n}  \\
&\quad +\left((3-n)t^2- a_1a_1'(n+2\alpha+2 )t+a_1^2(2n+2\alpha-1)\right)tw_{n}'    \notag \\ \notag \\
\end{align*}

\begin{align*}
0
 &= t^2(t^2-a_1^2)w_{n}''   \\
 &\quad + (n+1)a_1''t^2 w_{n+1}   \\
&\quad +\Big(\left(\left(-2n- \left((a_1')^2+a_1a_1''\right)(n+2\alpha )\right)t+a_1a_1'(3n+2\alpha )\right)t\\
&\qquad+(n+1)
 \left(nt^2+ a_1'\left(a_1'(n+2\alpha )t^2-2a_1(n+\alpha)t\right)\right)\Big)w_{n}  \\
&\quad +\left((3-n)t^3- a_1a_1'(n+2\alpha+2 )t^2+a_1^2(2n+2\alpha-1)t+(n+1)\left(a_1 a_1't^2-t^3\right)\right)w_{n}'    \notag \\ \notag \\
 &= t^2(t^2-a_1^2)w_{n}''     + (n+1) a_1''t^2w_{n+1}   \\
&\quad +\Big(\left(n(n-1)+ \left(n(a_1')^2-a_1a_1''\right)(n+2\alpha )\right)t^2 - n a_1a_1'
 \left( 2n+2\alpha-1\right)t\Big)w_{n}  \\
&\quad +\left(2(1-n)t^3- a_1a_1'( 2\alpha+1 )t^2+a_1^2(2n+2\alpha-1)t \right)w_{n}'
\end{align*}

We multiply this by $a_1't-a_1$ in order to use\eqnref{onlywn+1} to eliminate $w_{n+1}$:
\begin{align*}
0
 &= t^2(t^2-a_1^2)w_{n}'' (a_1't-a_1)    + (n+1)a_1''t^2(a_1't-a_1)w_{n+1}   \\
&\quad +\Big(\left(n(n-1)+ \left(n(a_1')^2-a_1a_1''\right)(n+2\alpha )\right)t^2 - n a_1a_1'
 \left( 2n+2\alpha-1\right)t\Big)(a_1't-a_1)w_{n}  \\
&\quad +\left(2(1-n)t^3- a_1a_1'( 2\alpha+1 )t^2+a_1^2(2n+2\alpha-1)t \right)(a_1't-a_1)w_{n}'
 \notag \\ \notag \\
 &= t^2(t^2-a_1^2)(a_1't-a_1)w_{n}''     +   a_1''t^2\Big( -t(t^2-a_1^2)w_{n}'   \notag \\
&\quad -\left(-nt^2- a_1a_1'(n+2\alpha )t+2a_1^2(n+\alpha)\right)w_{n} \Big)  \\
&\quad +\Big(\left(n(n-1)+ \left(n(a_1')^2-a_1a_1''\right)(n+2\alpha )\right)t^2 - n a_1a_1'
 \left( 2n+2\alpha-1\right)t\Big)(a_1't-a_1)w_{n}  \\
&\quad +\left(2(1-n)t^3- a_1a_1'( 2\alpha+1 )t^2+a_1^2(2n+2\alpha-1)t \right)(a_1't-a_1)w_{n}'
 \notag \\ \notag \\
 &= t^2(t^2-a_1^2)(a_1't-a_1)w_{n}''     - a_1''t^3 (t^2-a_1^2)w_{n}'   \notag \\
&\quad - a_1''t^2\left(-nt^2- a_1a_1'(n+2\alpha )t+2a_1^2(n+\alpha) \right)w_{n}    \\
&\quad +\Big(\left(n(n-1)+ \left(n(a_1')^2-a_1a_1''\right)(n+2\alpha )\right)t^2 - n a_1a_1'
 \left( 2n+2\alpha-1\right)t\Big)(a_1't-a_1)w_{n}  \\
&\quad +\left(2(1-n)t^3- a_1a_1'( 2\alpha+1 )t^2+a_1^2(2n+2\alpha-1)t \right)(a_1't-a_1)w_{n}'
 \notag \\ \notag \\
 &= t^2(t^2-a_1^2)(a_1't-a_1)w_{n}''    \notag \\
&\quad +\Big(  -   a_1''t^3 (t^2-a_1^2)\notag \\
&\qquad+ \left(2(1-n)t^3- a_1a_1'( 2\alpha+1 )t^2+a_1^2(2n+2\alpha-1)t \right)(a_1't-a_1)\Big)w_{n}'
 \notag \\
 &\quad +\Big(- a_1''t^2\left(-nt^2- a_1a_1'(n+2\alpha )t+2a_1^2(n+\alpha) \right)     \\
& +\Big(\left(n(n-1)+ \left(n(a_1')^2-a_1a_1''\right)(n+2\alpha )\right)t^2 - n a_1a_1'
 \left( 2n+2\alpha-1\right)t\Big)(a_1't-a_1)\Big)w_{n} \notag
 \end{align*}

Recall
 $$
\sum_{n\geq 0} w_n(t)z^{n}= \frac{b_0}{(1-2a_1z+t^2z^2)^\alpha}=\sum_{n\geq 1}b_{n-1}z^{n-1}=\sum_{n\geq 0}b_{n}z^{n}
 $$

Setting $\alpha =1$ we get the differential equation that the $b _n$ satisfy:

\begin{align}
0\label{tsquaredcase}
 &= t^2(t^2-a_1^2)(a_1't-a_1)y''      \\
&\quad +\Big(  -   a_1''t^3 (t^2-a_1^2) + \left(2(1-n)t^3- 3a_1a_1't^2+a_1^2(2n+1)t \right)(a_1't-a_1)\Big)y'
 \notag \\
 &\quad +\Big(- a_1''t^2\left(-nt^2- a_1a_1'(n+2 )t+2a_1^2(n+1) \right)     \notag \\
&\qquad +\Big(\left(n(n-1)+ \left(n(a_1')^2-a_1a_1''\right)(n+2 )\right)t^2 - n a_1a_1'
 \left( 2n+1\right)t\Big)(a_1't-a_1)\Big)y\notag
 \end{align}
where $\displaystyle{a_1(t):=\frac{t^2-1}{\sqrt{2(\beta-1)}}}$.  Note that $\displaystyle{t^2-a_1^2=2\frac{p}{\beta+1}}$.

After dividing this by $t$ and simplifying, this differential equation becomes
\begin{align}\label{firsttsquaredode}
0&=t \left(t^2+1\right) \left(t^4-2 \beta  t^2+1\right)y''  \\
&\quad -\left((2 n-3) t^6+t^4 (-4 \beta  n+2 n-5)+t^2(4 \beta -4 \beta  n+2 n+3)+2 n+1 \right)y'
 \notag \\
 &\quad -2\left(2 n t^5+n t^3 (\beta +(\beta +1) n+5)+n t (-\beta +(\beta +1) n+1)\right)y.\notag
\end{align}
Thus we have written the above differential equation in the form
$P^1(t)y''+Q_n(t)y'+R_n(t)y=0$, with
 \begin{align*}
 P^1(t)&=t \left(t^2+1\right) \left(t^4-2 \beta  t^2+1\right), \\
 Q_n(t)&= -\left((2 n-3) t^6+t^4 (-4 \beta  n+2 n-5)+t^2(4 \beta -4 \beta  n+2 n+3)+2 n+1 \right) \\
 &=-2 (n-1) \left(t^2+1\right) \left(t^4-2 \beta  t^2+1\right)+t^6+(4 \beta +3) t^4-5 t^2 -3 ,\\
 R_n(t)&=-2n\left(2  t^5+ t^3 (\beta +(\beta +1) n+5)+ t (-\beta +(\beta +1) n+1)\right).
 \end{align*}
 Note $Q_{n+1}(t)=Q_n(t)-2P^1t^{-1}$.
 The above derivation required that $\alpha\not \in\mathbb Z$.  We will show that the $b_n$ satisfy the above differential equation \eqnref{firsttsquaredode} in the next proposition.\end{proof}

 \begin{proposition}
The polynomials $b_n(t)$ satisfy \eqnref{firsttsquaredode} for all $n\geq 1$.
 \end{proposition}
 \begin{proof} We prove this by induction on $n\geq 0$.  We have the recursion relation \eqnref{recursionreln} so that for $\alpha =1$ we get for $n\geq 1$
 \begin{equation}
 b_{n+1}(t)-2a_1(t) b_{n}(t)+ t^2b_{n-1}(t)=0.
\end{equation}
Then
\begin{align}
b_{n+1}'&=2(a_1' b_{n}+a_1 b_{n}')-2tb_{n-1} -t^2b_{n-1}' \\
b_{n+1}''&=2(a_1''b_{n}+2a_1' b_{n}'+a_1b_n'')-2b_{n-1}-4tb_{n-1}' -t^2b_{n-1}'' .
\end{align}
Thus
 \begin{align*}
 P^1(t)b_{n+1}''&+Q_{n+1}b_{n+1}'\\ &= P^1\left(2(a_1''b_{n}+2a_1' b_{n}'+a_1b_n'')-2b_{n-1}-4tb_{n-1}' -t^2b_{n-1}''\right) \\
 &\quad +Q_{n+1}\left(2(a_1' b_{n}+a_1 b_{n}')-2tb_{n-1}-t^2b_{n-1}' \right)  \\ \\
&= 2P^1a_1''b_{n}+4P^1a_1' b_{n}'+2a_1P^1b_n''-2P^1b_{n-1}-4tP^1b_{n-1}' -t^2P^1b_{n-1}'' \\
 &\quad +2Q_{n+1}a_1' b_{n}+2Q_{n+1}a_1 b_{n}'-2tQ_{n+1}b_{n-1} -t^2Q_{n+1}b_{n-1}'  \\ \\
&= 2P^1a_1''b_{n}+4P^1a_1' b_{n}'+2a_1P^1b_n''-2P^1b_{n-1}-4tP^1b_{n-1}' -t^2P^1b_{n-1}'' \\
 &\quad +2Q_{n+1}a_1' b_{n}+2(Q_{n}-2P^1t^{-1})a_1 b_{n}'-2tQ_{n+1}b_{n-1} \\
 &\qquad  -t^2(Q_{n-1}-4P^1t^{-1})b_{n-1}'  \\  \\
&= 2P^1a_1''b_{n}+4P^1a_1' b_{n}' -2P^1b_{n-1}-4tP^1b_{n-1}'  \\
 &\quad +2Q_{n+1}a_1' b_{n}-4P^1t^{-1}a_1 b_{n}'-2tQ_{n+1}b_{n-1} \\
 &\qquad  +4t^2P^1t^{-1}b_{n-1}'-2R_na_1b_n+t^2R_{n-1}b_{n-1}  \\  \\
 &= (2P^1a_1'' +2Q_{n+1}a_1'   -2R_na_1)b_n\\
 &\quad +4P^1a_1' b_{n}'   -4P^1t^{-1}a_1 b_{n}' \\
 &\qquad  -(2P^1+2tQ_{n+1}-t^2R_{n-1})b_{n-1}  \\  \\
&= -R_{n+1}b_{n+1}-4 n a'_1 (1 + t^2)^2 (t^2 - \beta)b_n\\
 &\quad +4P^1a_1' b_{n}'   -4P^1t^{-1}a_1 b_{n}' \\
 &\qquad  +4 (n+1) t \left(t^2+1\right)^3b_{n-1}  \\
 \end{align*}
 Thus the proof reduces to proving showing the terms after $-R_{n+1}b_{n+1}$ vanish for all $n\geq 1$.
 Set
 \begin{align*}
 \text{Rem}(t,n)&:=-4 n a'_1 (1 + t^2)^2 (t^2 - \beta)b_n+4P^1(a_1'-t^{-1}a_1) b_{n}' \\
 &\quad  +4 (n+1) t \left(t^2+1\right)^3b_{n-1}.
 \end{align*}
 If one expands out
 \begin{equation}\label{finalreduction}
 \text{Rem}(t,n+1)-2a(t)\text{Rem}(t,n)+t^2\text{Rem}(t,n-1)
 \end{equation}
 using \eqnref{recursionreln} and \eqnref{recursionrelndir}, then \eqnref{finalreduction} reduces to zero.   Hence  by induction we have proven that the $b_n$ satisfy \eqnref{firsttsquaredode}.
 \end{proof}

In order to put this into Sturm-Louiville like form we first divide by $t \left(t^2+1\right) \left(t^4-2 \beta  t^2+1\right)$, and then multiply the differential equation above by the integrating factor
\begin{align*}
\exp\int \frac{Q_n(t)}{P^1(t)}\,dt  &=\frac{\left(t^4-2 \beta
   t^2+1\right)^{3/2}}{\left(t^2+1\right) t^{2 n+1}}.
   \end{align*}
   We obtain
   \begin{equation}\label{niceform}
  \frac{d}{dt}\left(\frac{\left(t^4-2 \beta
   t^2+1\right)^{3/2}}{\left(t^2+1\right) t^{2 n+1}}\frac{dy}{dt}\right)+\frac{\left(t^4-2 \beta
   t^2+1\right)^{1/2}R_n(t)}{(t^2+1)^2t^{2n+2}}y=0.
   \end{equation}
   If we set
   $$
\frac{d}{d\tau_n}:=\frac{\left(t^4-2 \beta
   t^2+1\right)^{3/2}}{\left(t^2+1\right) t^{2 n+1}}\frac{d}{dt}
   $$
    we get
   $$
  \tau_n(t)=\int \frac{\left(t^2+1\right) t^{2 n+1}}{\left(t^4-2 \beta
   t^2+1\right)^{3/2}}\, dt,
   $$
which is an elliptic integral and can be expressed in the form of a series about $t=0$. We will not write this out as it doesn't seem to simplify matters.
Then \eqnref{niceform} becomes
\begin{equation*}
 \frac{\left(t^2+1\right) t^{2 n+1}}{\left(t^4-2 \beta
   t^2+1\right)^{3/2}}\frac{d^2y}{d\tau_n^2}+\frac{\left(t^4-2 \beta
   t^2+1\right)^{1/2}R_n(t)}{(t^2+1)^2t^{2n+2}}y=0.
   \end{equation*}
or
\begin{equation} \label{niceform2}\boxed{
\frac{d^2y}{d\tau_n^2}+\frac{\left(t^4(\tau_n)-2 \beta
   t^2(\tau_n)+1\right)^{2}R_n(t(\tau_n))}{(t^2(\tau_n)+1)^3t^{4n+3}(\tau_n)}y=0.}
   \end{equation}

 \begin{theorem} Define polynomials $c_n,d_{n-1} $ by $c_n+d_{n}\sqrt{P}=\lambda_1^n$ for $n\in\mathbb N$. Then
  $d_n(t)$ satisfies $$P^2(t)y''+Q_n^2(t)y'+R_n^2(t)y=0$$
  with $$\aligned
P^2(t)&=t(1-t^2)(t^4-2\beta t^2+1),\\
Q^2_n(t)&=(2 n-3) t^6-t^4 (4 \beta  n+2 n-5)+t^2 (-4 \beta +4 \beta  n+2 n+3)-2 n-1, \\
R_n^2(t)&=2  n t \left(\beta -(\beta -1) n+t^2 (\beta +(\beta -1) n-5)+2 t^4+1\right)\endaligned$$
 \end{theorem}

\begin{proof}
We next will derive  second order linear differential equations satisfied by $c_n$ and $d_n$ where are defined by $c_n+d_{n}\sqrt{p(t)}=(c_1+d_1\sqrt{p})^n$, for
\begin{equation}\label{secondtsquaredcondition}
c_0=1,\quad d_0=0,\quad c_1=\frac{t^2+1}{\sqrt{2(\beta+1)}}\quad d_1=\sqrt{\frac{\beta-1}{2}}.
\end{equation}
First we have
\begin{align*}
\sum_{n\geq 0}c_nz^n+\sum_{n\geq 0}d_{n}z^n\sqrt{p}
&=\sum_{n\geq 0}(c_1+d_1\sqrt{p})^nz^n  \\
&=\sum_{n\geq 0}(\lambda_2(\beta,t))^nz^n  \\
&=\sum_{n\geq 0}(-\imath)^n(\lambda_1(-\beta,\imath t))^nz^n \\
&=\sum_{n\geq 0}(-\imath)^n(a_1(-\beta,\imath t)+b_1(-\beta,\imath t)\sqrt{p})^nz^n \\
&=\sum_{n\geq 0}(-\imath)^n(a_n(-\beta,\imath t)+b_n(-\beta,\imath t)\sqrt{p})z^n
\end{align*}
by the following formula
\begin{align}\label{second duality}
-\imath \lambda_1(-\beta,\imath t)&=-\imath \left(\frac{-t^2+1}{\sqrt{2(-\beta+1)}}+\sqrt{\frac{-\beta-1}{2}}\sqrt{p}\right)  \\
&=  \frac{t^2-1}{\sqrt{2(\beta-1)}}+\sqrt{\frac{\beta+1}{2}}\sqrt{p} \notag \\
&=\lambda_2(\beta,t),\notag
\end{align} where $\imath=\sqrt{-1}$.
Thus
\begin{align*}
c_n(\beta,t)=(-\imath)^na_n(-\beta,\imath t),\quad d_n(\beta,t)=(-\imath)^nb_n(-\beta,\imath t).
\end{align*}
We then set
\begin{align*}
P^2(t)&:=t(1-t^2)(t^4-2\beta t^2+1) \\
Q^2_n(t)&=(2 n-3) t^6-t^4 (4 \beta  n+2 n-5)+t^2 (-4 \beta +4 \beta  n+2 n+3)-2 n-1 \\
R_n^2(t)&=2  n t \left(\beta -(\beta -1) n+t^2 (\beta +(\beta -1) n-5)+2 t^4+1\right)
\end{align*}
 Consequently the $d_n(\beta,t)$ satisfy the second order linear differential equation
\begin{equation}
P^2(t)y''+Q_n^2(t)y'+R_n^2(t)y=0.
\end{equation}
\end{proof}

\section{Acknowledgements}  The two authors would like to thank the Mittag-Leffler Institute in Djursholm, Sweden for its hospitality and great working environment where part of this work was done.  The first author is partially supported by a collaboration grant from the Simons Foundation (\#319261).   The second author is partially supported by NSF of China (Grant 11271109) and NSERC.

\begin{thebibliography}{Sch03b}
\bibitem[BJ05]{MR2203507}
Mikhail Belolipetsky and Gareth~A. Jones.
\newblock Automorphism groups of {R}iemann surfaces of genus {$p+1$}, where
  {$p$} is prime.
\newblock {\em Glasg. Math. J.}, 47(2):379--393, 2005.


\bibitem[Bre00]{MR1796706}
Thomas Breuer.
\newblock {\em Characters and automorphism groups of compact {R}iemann
  surfaces}, volume 280 of {\em London Mathematical Society Lecture Note
  Series}.
\newblock Cambridge University Press, Cambridge, 2000.

\bibitem[CFT13]{MR3090080}
Ben Cox, Vyacheslav Futorny, and Juan~A. Tirao.
\newblock D{JKM} algebras and non-classical orthogonal polynomials.
\newblock {\em J. Differential Equations}, 255(9):2846--2870, 2013.


\bibitem[CGLZ1]{CGLZ1} B. Cox, X. Guo,  R. Lu,  K. Zhao,  B. Cox, X. Guo,  R. Lu,  K. Zhao,  n-Point Virasoro algebras and their modules of densities, {\em Commun. Contemp. Math.}  16  (2014),  no. 3, 1350047.

\bibitem[CGLZ2]{CGLZ2} B. Cox, X. Guo,  R. Lu,  K. Zhao,  Simple superelliptic Lie algebras,  arXiv:1412.7777.

\bibitem[DJKM]{DJKM} E. Date, M. Jimbo, M. Kashiwara,T. Miwa, Landau-Lifshitz equation: solitons, quasi periodic solutions and infinite-dimensional Lie algebras, J. Phys. A16(2)(1983), 221–236.

\bibitem[DS04]{MR2183270}
Art{\=u}ras Dubickas and J{\"o}rn Steuding.
\newblock The polynomial {P}ell equation.
\newblock {\em Elem. Math.}, 59(4):133--143, 2004.


\bibitem[FBZ01]{MR1849359}
Edward Frenkel and David Ben-Zvi.
\newblock {\em Vertex algebras and algebraic curves}, volume~88 of {\em
  Mathematical Surveys and Monographs}.
\newblock American Mathematical Society, Providence, RI, 2001.

\bibitem[Jor86]{MR829385}
D.~A. Jordan.
\newblock On the ideals of a {L}ie algebra of derivations.
\newblock {\em J. London Math. Soc. (2)}, 33(1):33--39, 1986.

\bibitem[KL91]{MR1104840}
David Kazhdan and George Lusztig.
\newblock Affine {L}ie algebras and quantum groups.
\newblock {\em Internat. Math. Res. Notices}, (2):21--29, 1991.

\bibitem[KN87a]{MR925072}
Igor~Moiseevich Krichever and S.~P. Novikov.
\newblock Algebras of {V}irasoro type, {R}iemann surfaces and strings in
  {M}inkowski space.
\newblock {\em Funktsional. Anal. i Prilozhen.}, 21(4):47--61, 96, 1987.

\bibitem[KN87b]{MR902293}
Igor~Moiseevich Krichever and S.~P. Novikov.
\newblock Algebras of {V}irasoro type, {R}iemann surfaces and the structures of
  soliton theory.
\newblock {\em Funktsional. Anal. i Prilozhen.}, 21(2):46--63, 1987.

\bibitem[KN89]{MR998426}
Igor~Moiseevich Krichever and S.~P. Novikov.
\newblock Algebras of {V}irasoro type, the energy-momentum tensor, and operator
  expansions on {R}iemann surfaces.
\newblock {\em Funktsional. Anal. i Prilozhen.}, 23(1):24--40, 1989.

\bibitem[Sch03a]{MR2058804}
Martin Schlichenmaier.
\newblock Higher genus affine algebras of {K}richever-{N}ovikov type.
\newblock {\em Mosc. Math. J.}, 3(4):1395--1427, 2003.

\bibitem[Sch03b]{MR1989644}
Martin Schlichenmaier.
\newblock Local cocycles and central extensions for multipoint algebras of
  {K}richever-{N}ovikov type.
\newblock {\em J. Reine Angew. Math.}, 559:53--94, 2003.

\bibitem[Sha03]{MR2035219}
Tanush Shaska.
\newblock Determining the automorphism group of a hyperelliptic curve.
\newblock In {\em Proceedings of the 2003 {I}nternational {S}ymposium on
  {S}ymbolic and {A}lgebraic {C}omputation}, pages 248--254 (electronic). ACM,
  New York, 2003.

\bibitem[She03]{MR2072650}
O.~K. She{\u\i}nman.
\newblock Second-order {C}asimirs for the affine {K}richever-{N}ovikov algebras
  {$\widehat{\mathfrak g\mathfrak l}_{g,2}$} and {$\widehat{\mathfrak
  s\mathfrak l}_{g,2}$}.
\newblock In {\em Fundamental mathematics today ({R}ussian)}, pages 372--404.
  Nezavis. Mosk. Univ., Moscow, 2003.

\bibitem[She05]{MR2152962}
O.~K. Sheinman.
\newblock Highest-weight representations of {K}richever-{N}ovikov algebras and
  integrable systems.
\newblock {\em Uspekhi Mat. Nauk}, 60(2(362)):177--178, 2005.

\bibitem[Skr88]{MR966871}
S.~M. Skryabin.
\newblock Regular {L}ie rings of derivations.
\newblock {\em Vestnik Moskov. Univ. Ser. I Mat. Mekh.}, (3):59--62, 1988.

\bibitem[SS98]{MR1666274}
M.~Schlichenmaier and O.~K. Scheinman.
\newblock The {S}ugawara construction and {C}asimir operators for
  {K}richever-{N}ovikov algebras.
\newblock {\em J. Math. Sci. (New York)}, 92(2):3807--3834, 1998.
\newblock Complex analysis and representation theory, 1.

\bibitem[SS99]{MR1706819}
M.~Shlichenmaier and O.~K. Sheinman.
\newblock The {W}ess-{Z}umino-{W}itten-{N}ovikov theory,
  {K}nizhnik-{Z}amolodchikov equations, and {K}richever-{N}ovikov algebras.
\newblock {\em Uspekhi Mat. Nauk}, 54(1(325)):213--250, 1999.

\end{thebibliography}

\def\cprime{$'$}

\end{document}